\documentclass[11pt,leqno]{amsart}

\usepackage[letterpaper,margin=1.17in,footskip=0.27in]{geometry}
\usepackage[T1]{fontenc}
\IfFileExists{scholax.sty}{%
  \usepackage[p,scosf,looser]{scholax}
  \usepackage[scaled=1.08,ncf,vvarbb]{newtxmath}
}{%
\usepackage{mathrsfs}
  \usepackage{lmodern}
  \usepackage{mathpazo}
  \usepackage{amssymb}
}
\usepackage{mathtools}
\usepackage{aliascnt}
\usepackage{microtype}
\usepackage{enumitem}
\usepackage[dvipsnames]{xcolor}
\usepackage[colorlinks,pagebackref]{hyperref}
\usepackage[nameinlink,noabbrev]{cleveref}
\usepackage{bookmark}
\usepackage{hyperref}
\usepackage{cleveref}
\usepackage{bookmark}
\hypersetup{citecolor=ForestGreen,linkcolor=red,urlcolor=MidnightBlue}
\allowdisplaybreaks
\numberwithin{equation}{section}
\newtheorem{theorem}{Theorem}[section]
\newtheorem{lemma}[theorem]{Lemma}
\newtheorem{proposition}[theorem]{Proposition}
\newtheorem{example}[theorem]{Example}
\newtheorem{corollary}[theorem]{Corollary}
\newtheorem{claim}{Claim}

\theoremstyle{definition}
\newtheorem{definition}[theorem]{Definition}
\theoremstyle{remark}
\newtheorem{remark}[theorem]{Remark}
\newtheorem{question}[theorem]{Question}
\newcommand{\Q}{\mathbb{Q}}
\newcommand{\Z}{\mathbb{Z}}

\newcommand{\cC}{\mathscr{C}}
\newcommand{\cD}{\mathscr{D}}
\newcommand{\cE}{\mathscr{E}}
\newcommand{\cA}{\mathscr{A}}
\newcommand{\cO}{\mathcal{O}}
\newcommand{\tauB}{\tau_{\mathrm b}}
\newcommand{\st}{\mathbin{\star}}
\newcommand{\ceil}[1]{\left\lceil #1\right\rceil}

\newcommand{\stacktag}[1]{\href{https://stacks.math.columbia.edu/tag/#1}{#1}}
\DeclareMathOperator{\Spec}{Spec}
\DeclareMathOperator{\Proj}{Proj}
\DeclareMathOperator{\Hom}{Hom}

\DeclareMathOperator{\Frac}{Frac}
\DeclareMathOperator{\im}{im}
\DeclareMathOperator{\Tr}{Tr}

\allowdisplaybreaks
\numberwithin{equation}{section}
\title{Test ideals can be computed by boundaries}
\author{Rahul Ajit}
\address{Department of Mathematics, The University of Kansas, 405 Snow Hall, Lawrence, KS 66045}
\email{rahulajit@ku.edu, confusedrahul@yahoo.com}
\dedicatory{With an appendix joint with Daniel Apsley.}
\date{\today}
\begin{document}

\begin{abstract}
Let $R$ be a normal $F$-finite ring of characteristic $p>0$. We show
that \textit{one} effective $\Q$-divisor $\Delta$, with $K_R+\Delta$
$\Q$-Cartier of index prime to $p$, computes the big
test ideal. This answers a question of
Schwede and Smith. More generally, we prove their conjecture that test module of a Cartier
algebra acting on an ideal is the test module of a
single map $\psi$ in the algebra and even the image of $\psi$. We prove a
global version for schemes quasi-projective over an $F$-finite local ring as well. In particular, every such normal, locally
strongly $F$-regular scheme admits an effective boundary $\Delta$
for which $(X,\Delta)$ is locally strongly $F$-regular and
$K_X+\Delta$ is $\Q$-Cartier of index prime to $p$.
\end{abstract}
\maketitle
\tableofcontents

\section{Introduction}

Test ideals originated from tight closure theory \cite{HochsterHuneke1990}.
They are characteristic $p$ analogs of multiplier ideals
\cite{Smith2000,Hara2001,HaraYoshida2003,Takagi2004}. (Big) test ideal
$\tauB(X)$ of a normal $F$-finite scheme $X$ needs no Cartier assumption on $K_X$
\cite{LyubeznikSmith2001,Schwede2011}. But we know many of its good properties mainly for pairs $(X,\Delta)$ with
$K_X+\Delta$ $\Q$-Cartier. For example, if $X$ is integral, separated and $D\geq0$ Cartier,  one alteration computes $\tauB(X,\Delta+\lambda D)$ for all
$\lambda\geq0$. $F$-jumping numbers of this family are discrete and
rational \cite[Theorems~3.1 and~4.1]{SchwedeTuckerZhang2012}, see also
\cite{BlickleSchwedeTakagiZhang2010}. Both also hold for family
$\tauB(X;\cO_X(-D)^\lambda)$ when anti-canonical algebra of such $X$ is
finitely generated \cite[Theorem~5.6 and Theorem~B]{ChiecchioEnescuMillerSchwede2018}.
Without such hypotheses, $F$-jumping numbers of $\tauB(R;\mathfrak m^t)$ can be
irrational \cite[Theorem~2]{AjitJumping2026}.

Since $\tauB(X,\Delta)\subseteq\tauB(X)$ for every boundary $\Delta$
\cite[Proposition~3.23(i)]{Schwede2011}, it is natural to look for a single boundary
$\Delta$, with $K_X+\Delta$ $\Q$-Cartier, giving equality. For affine
$X$ this was asked in \cite[Question~7.3]{SchwedeSmith2010}, also see \cite[Question~6.3]{Schwede2011}. A related question
asks for $(X,\Delta)$ locally strongly $F$-regular when $X$ is so
\cite[Question~7.2]{SchwedeSmith2010}. Schwede and Smith proved the affine case \cite[Corollary~6.9]{SchwedeSmith2010} but non-affine case was open.

For a normal quasi-projective variety $X$ over an algebraically closed field
of characteristic $0$, analogous statements are known. Multiplier ideal of $X$ is the multiplier ideal of a pair $(X,\Delta)$ with
$K_X+\Delta$ $\Q$-Cartier. $X$ is log terminal in the sense of
\cite{deFernexHacon2009} iff some such pair is klt
\cite[Theorems~1.1 and~1.2]{deFernexHacon2009}. In characteristic $p$, Schwede showed that big
test ideal of a normal $F$-finite domain is a finite sum
$\sum_i\tauB(R,\Delta_i)$ with all $K_R+\Delta_i$ $\Q$-Cartier
\cite[Corollary~5.4]{Schwede2011}. But a sum of ideals is in general not one
of the summands. So this does not give one boundary.

It is easier to think in terms of maps. A nonzero $p^{-e}$-linear map
$\phi:F_*^eR\to R$ gives a boundary $\Delta_\phi$ with
$(p^e-1)(K_R+\Delta_\phi)\sim0$. We get $\tauB(R,\Delta_\phi)=\tauB(R;R\langle\phi\rangle)$ for the algebra $R\langle\phi\rangle$ \cite[Theorem~3.11]{SchwedeFAdjunction2009},
\cite[Remark~4.3 and Corollary~4.8]{Schwede2011}. $\Delta_i$ above come from such
maps. More generally, test module $\tau(I,\mathscr B)$ of a Cartier algebra
$\mathscr B$ acting on an ideal $I$ of a reduced ring $R$ is a finite sum of
test modules $\tau(I,R\langle\phi\rangle)$ of single maps $\phi\in\mathscr B$
\cite[Chapter~5, Theorem~8.31]{SchwedeSmithBook2026}.
\cite[Chapter~5, Conjecture~8.33]{SchwedeSmithBook2026} conjectured that one map is enough. Our main result proves this conjecture (see
Section~\ref{sec:test-modules} for definitions). However, we will not pick a good summand. Instead we build a new map which does the job all of them at once.

\begin{theorem}\label{thm:main}
Let $R$ be a reduced $F$-finite ring of characteristic $p>0$,
$\mathscr B$ be a Cartier algebra over $R$ acting on an ideal $I\subseteq R$.
Then there are $s>0$ and $\psi\in\mathscr B_s$ such that 
\[
 \tau(I,\mathscr B)=\tau\bigl(I,R\langle\psi\rangle\bigr)=\psi(F_*^sI)=\psi(F_*^s \tau(I,\mathscr B)).
\]
\end{theorem}

So $\tau(I,\mathscr B)$ is the test module of one map and also its image. Every Cartier
subalgebra $\cD$ of $\cC(R)=\bigoplus_{e\geq0}\Hom_R(F_*^eR,R)$ acts on $I=R$. If $R$ is a normal domain and
$\cD_e\neq0$ for some $e>0$, test module of $\cD$ is the big test ideal
$\tauB(R;\cD)$. Theorem~\ref{thm:main} gives:

\begin{corollary}\label{thm:intro-principalization}
Let $R$ be a normal $F$-finite domain, $\cD\subseteq\cC(R)$ be a
Cartier subalgebra with $\cD_e\neq0$ for some $e>0$. Then there are $s>0$ and
$0\neq\psi\in\cD_s$ such that
\[
 \tauB(R;\cD)=\tauB\bigl(R;R\langle\psi\rangle\bigr)
             =\tauB(R,\Delta_\psi).
\]
Moreover, $\Delta_\psi\geq0$ and $(p^s-1)(K_R+\Delta_\psi)\sim0$.
\end{corollary}

Applied to the full Cartier algebra on each component of $\Spec R$, we
answer \cite[Question~7.3]{SchwedeSmith2010}.

\begin{corollary}\label{thm:intro-affine}
Let $R$ be a normal $F$-finite ring. There is an effective $\Q$-divisor
$\Delta$ on $X=\Spec R$ with $\tauB(R,\Delta)=\tauB(R)$ and
$(p^s-1)(K_X+\Delta)\sim0$ for some $s>0$.
\end{corollary}

For non-affine $X$ we need one global divisor. Local boundaries above need not
glue. So we run the affine argument homogeneously, on a section ring of a
projective compactification. As in \cite[Theorem~4.3]{SchwedeSmith2010} we
work over an $F$-finite local\footnote{our proof never uses that the base is local} ring, but unlike there we do not need any assumption on ($F$-)singularities of $X$.

\begin{theorem}\label{thm:intro-global}
Let $X$ be a normal Noetherian scheme quasi-projective over an $F$-finite
local ring. There is an effective $\Q$-divisor $\Delta$ such that
$\tauB(X,\Delta)=\tauB(X)$ with $K_X+\Delta$
$\Q$-Cartier of index prime to $p$.

If $X$ is integral and $\dim X>0$, one can choose a normal projective
compactification $X\subseteq Y$, an ample Cartier divisor $H$ on $Y$, an effective $\Q$-divisor $\Delta_Y$ with $\Delta=\Delta_Y|_X$, such
that $\tauB(Y,\Delta_Y)=\tauB(Y)$ and $(p^s-1)(K_Y+\Delta_Y)\sim aH$ for
some $s>0$, $a > 0$.
\end{theorem}

Test ideal detects strong $F$-regularity. So this answers
\cite[Question~7.2]{SchwedeSmith2010} in our setting. In fact $X$ and
$(X,\Delta)$ have the same strongly $F$-regular locus.

\begin{corollary}\label{cor:strong-locus}
Let $X$ satisfy the hypotheses of Corollary~\ref{thm:intro-affine} or
Theorem~\ref{thm:intro-global}. Choose $\Delta$ as there. For every
$x\in X$, the local ring $\cO_{X,x}$ is strongly $F$-regular if and only
if the pair $(\cO_{X,x},\Delta|_{\Spec\cO_{X,x}})$ is. In particular, if
$X$ is locally strongly $F$-regular, then so is $(X,\Delta)$.
\end{corollary}

\begin{proof}
Equality of test ideals localizes
\cite[Proposition~3.23(ii)]{Schwede2011}. A pair is strongly
$F$-regular iff its test ideal is trivial
\cite[Proposition~3.23(iii)]{Schwede2011}.
\end{proof}

For a normal domain $R$, the image part in Theorem~\ref{thm:main} gives
a bit more: we can choose $\Delta$ in Corollary~\ref{thm:intro-affine} such
that $(\Spec R_P,\Delta)$ is sharply $F$-pure exactly when $R_P$ is strongly
$F$-regular (Remark~\ref{rem:one-image}).

Section~\ref{sec:consequences} transfers to $\tauB(X)$ some important results known for
pairs with $K_X+\Delta$ $\Q$-Cartier. We first treat mixed test ideals
(Corollary~\ref{cor:intro-mixed}). Then we express $\tauB(X)$ as a trace image from one
finite separable cover (Corollary~\ref{cor:one-cover}, via
\cite[Theorem~4.1]{SchwedeTuckerZhang2012}). We end this note with Appendix \ref{App}, joint with Daniel Apsley, where we answer a question of Polstra and Tucker \cite[Question~7.5]{PolstraTucker2018} negatively on computing $F$-signature using a boundary. We also record 2 interesting examples that were found by AI.

Section~\ref{sec:background}
introduces the main players.
Section~\ref{sec:affine} treats normal domains and proves
Corollaries~\ref{thm:intro-principalization} and~\ref{thm:intro-affine}.
Section~\ref{sec:general} proves Theorem~\ref{thm:main}. Section~\ref{sec:projective} proves Theorem~\ref{thm:intro-global}.

\subsection{Idea of the proof}\label{sec:idea}

Main case is normal domain $R$, $I=R$ , $\cD\subseteq\cC(R)$ with
$\cD_1\neq0$. We explain it first. Write $\tauB(R;\cD)=(a_1,\ldots,a_m)$.
For every nonzero homogeneous $\psi\in\cD$ of positive degree,
$R\langle\psi\rangle\subseteq\cD$. So
$\tauB(R;R\langle\psi\rangle)\subseteq\tauB(R;\cD)$. Hence we need one map
$\psi$ with $a_1,\ldots,a_m\in\tauB(R;R\langle\psi\rangle)$. There are $2$
problems: to find one $\psi$ which takes all the values $a_i$ and to make
sure that these values lie in the test ideal of $\psi$ itself.

\emph{One map with given values.} This is linear algebra, much like Lagrange
interpolation. As in proof of \cite[Proposition~5.6]{Schwede2011}, we fix
$0\neq\theta\in\cD_1$, a test element $c$ such that $R_c$ is regular,
$\theta_c$ generates $\Hom_{R_c}(F_*R_c,R_c)$ and $c\in\theta(F_*(cR))$
(Lemma~\ref{lem:common-data}). If $\dim R>0$, then $\Frac R$ is not perfect. So rank of $F_*^eR$ goes to infinity with $e$ (Lemma~\ref{lem:nonperfect}).
For $e\gg0$, after replacing $c$ by a multiple, $(F_*^eR)_c$ is free of rank
$>m$. Clearing denominators in $m+1$ basis vectors and in their coordinate
functions we get $w_0,\ldots,w_m\in F_*^eR$ and $\beta_0,\ldots,\beta_m\in\cD_e$
with $\beta_i(w_j)=\delta_{ij}c^L$ for some $L\geq1$. On the other side, $c^L$ is a test element. By \eqref{eq:hdbs-formula}
each $a_i$ is a sum of values of maps in $\cD$ on $c^L$. Composing with
suitable multiples of $\theta$ moves all these maps to one degree $D$
(Lemma~\ref{lem:synchronization}). This gives $\alpha_i\in\cD_D$ with
$\alpha_i(F_*^Dc^L)=a_i$. Now $\psi_0=\sum_i\alpha_i\st\beta_i$ sends
$F_*^Dw_j$ to $a_j$. $\beta_j$ picks up the direction $w_j$. $\alpha_j$
puts the value $a_j$ there, like a Lagrange polynomial
(Proposition~\ref{prop:finite-interpolation}). Here $a_0$ is one extra value,
which we choose below.

\emph{Values in the test ideal.} A value of a map need not lie in the test
ideal of this map. By \eqref{eq:hdbs-formula} it does if $c^2$ is a test
element for $R\langle\psi\rangle$ and the input lies in $F_*^s(c^2R)$, where
$s=\deg\psi$. By a variant of \cite[Lemma~5.7]{Schwede2011}
(Lemma~\ref{lem:principal-test-element}), $c^2$ is such a test element as
soon as $c\in\psi(F_*^s(cR))$ and $(\Spec R_c,\Delta_\psi)$ is strongly
$F$-regular. Our $\psi_0$ has no reason to satisfy any of this. So we cheat a
little. We take $a_0=c$ and arrange $w_j\in F_*^e(c^2R)$. Then replace
$\psi_0$ by $\psi=\psi_0\st\theta$. Composing with $\theta$ loses no values on
$c^2R$, as $c^2R\subseteq\theta(F_*(c^2R))$. So $\psi$ still takes the values
$a_j$ on $F_*^s(c^2R)$. In particular $c\in\psi(F_*^s(cR))$. Composing
with $\theta$ makes $\psi$ strongly $F$-regular on $\Spec R_c$. There
$\psi_0$ is surjective because its image contains the unit $c$. So
$(\Spec R_c,\Delta_{\psi_0})$ is sharply $F$-pure. As $\theta_c$ is a
generator, Lemma~\ref{lem:product-divisor} gives
$\Delta_\psi=\lambda\Delta_{\psi_0}$ on $\Spec R_c$, with
$\lambda=p(p^{s_0}-1)/(p^{s_0+1}-1)<1$ and $s_0=\deg\psi_0$. On a regular
scheme, a sharply $F$-pure boundary times a rational $\lambda\in(0,1)$ is
strongly $F$-regular \cite[Lemma~4.9(iii)]{SchwedeSmith2010}. So $c^2$ is a
test element for $R\langle\psi\rangle$. All $a_j$ lie in
$\tauB(R;R\langle\psi\rangle)$ (Lemma~\ref{lem:compose-theta}). This idea to make the algebra of one map strongly $F$-regular by composing comes from
\cite[Theorem~4.3(i)]{SchwedeSmith2010}, see also
\cite[Section~6]{Schwede2011}. There one only needs
$1\in\tauB(R;R\langle\psi\rangle)$, while we need every $a_j$.

\emph{Reduced rings and ideals.} First, $\cD_1\neq0$ is not needed. An
iterate of a map has the same boundary, hence same test ideal. So
Veronese $\bigoplus_e\cD_{ne}$ has the same test ideal as $\cD$
(Lemma~\ref{lem:veronese}). So we can run the argument with $F^n$ in place
of $F$ (Corollary~\ref{cor:normal-principal}). Now let $R$ be reduced, 
$J=\tau(I,\mathscr B)\neq0$. Let $S$ be the normalization of
$R/\operatorname{Ann}_RJ$. All the $x\in J$ with $Sx\subseteq J$ form a
$\mathscr B$-submodule which agrees with $J$ at the minimal primes of $R$. So
by minimality it is all of $J$. Thus $J$ is an ideal of $S$. Every
map in $\mathscr B$ extends from $J$ to $S$
(Proposition~\ref{prop:normalize-J}). So normal case gives one map
$\theta$ on $S$, of some degree $e$, with test ideal $J$. But $\theta$ is a sum of extended maps with coefficients in $S$, not in $R$. Choose \(c\in J\) such that \(cS\subseteq R/\operatorname{Ann}_R J\). Then \(c\) clears the coefficients coming from \(S\). However, multiplying a map on the source by \(c\) may decrease its test ideal. So we first take an iterate $\theta^{[N]}$ which
maps $F_*^{Ne}(c^4S)$ onto $J$. Then $\eta=(F_*^{Ne}c^2)\cdot\theta^{[N]}$
still maps onto $J$ and has test ideal $J$. This $\eta$ comes from a map
$\psi\in\mathscr B$, which does the job.

\emph{Projective case.} Let $S$ be a standard graded normal $F$-finite domain,
$Y=\Proj S$, $U=\Spec S\setminus V(S_+)$. The argument above works with
homogeneous maps, but all prescribed values must have the same degree. This
is easy to arrange. If $b_i$'s are homogeneous generators of $\tauB(S)$ and
$x_j$ generate $S_1$, then for $N\gg0$ the elements $b_ix_j^{N-\deg b_i}$ all
have degree $N$, generating the same ideal as the $b_i$ on $U$.
(Powers of the $b_i$ would have the right support but the wrong ideal.) Resulting homogeneous $\psi$ has degree $a/p^s$ for some $a\in\Z$. It gives a
map $F_*^s\cO_Y(-aH)\to\cO_Y$, where $\cO_Y(H)\cong\cO_Y(1)$, hence a boundary
$\Delta_Y$ on $Y$. On a chart $D_+(f)=\Spec B$, $f\in S_1$, we have
$S_f=B[f,f^{-1}]$. Up to a unit $\psi_f$ is the Laurent extension of the
map it induces on $B$. Test ideals commute with Laurent extensions
(Lemma~\ref{lem:laurent-test-ideal}, cf.\
\cite[Theorem~4.1 and Application~5.2]{BravoSmith2002}). So equality of test
ideals on $U$ descends to $Y$ (Proposition~\ref{prop:Proj-descent}).

Our boundary is far from canonical. Our construction gives no control of
its behavior under
reduction from characteristic $0$ (compare \cite[Remark~6.2]{Schwede2011}). We may choose $a>0$ in Theorem~\ref{thm:intro-global}. So
$K_Y+\Delta_Y$ can be chosen ample. We cannot in general choose
$a<0$. So this does not give  log Fano type statement \cite{SchwedeSmith2010}.

\subsection{Conventions}

All rings are Noetherian of characteristic $p>0$ and $F$-finite, meaning
Frobenius is finite. For a reduced ring $R$, $R^\circ$ is the complement of
the union of its minimal primes. A boundary is an effective $\Q$-divisor. All
test ideals in this note are big test ideals. On a normal $F$-finite scheme
$X$ with a boundary $\Gamma$, $\tauB(X,\Gamma)$ is the ideal sheaf given by
the test ideal on each affine open. These glue since test ideals localize
\cite[Proposition~3.23(ii)]{Schwede2011}. For a normal integral $F$-finite scheme $X$, let $\omega_X$ be the canonical
module of \cite[Theorem~1.0.1 and Corollary~1.0.3]{BhattBlickleSchwedeTucker2026}. Choose a Weil  $K_X$ such that
$\cO_X(K_X)\cong\omega_X$. These choices are compatible with
restriction to open subschemes and satisfies
$\mathcal Hom_{\cO_X}(F_*\cO_X,\omega_X)\cong F_*\omega_X$. For $X=\Spec R$ we also write $K_R$. Finally,
$F$-finite rings are excellent \cite{Kunz1976}; see also \cite[Chapter~1, Theorem~1.23(b)]{SchwedeSmithBook2026}.

\subsection*{Acknowledgments}

I am extremely grateful to my advisors. Karl Schwede introduced me to these questions, taught me char $p$ patiently \& carefully and for countless stimulating discussions. Also \cite{Schwede2011} played fundamental role in this work. Christopher Hacon shared many geometric insights and our discussion surrounding \cite{deFernexHacon2009} over the years shaped several parts of the paper philosophically. I am grateful to 
Daniel Hernández \& Emily Witt for their support, stimulating mathematical discussions and interesting questions. In particular, a primitive version of \Cref{sec:affine} was written in the form of a letter to Daniel Hernández. Special thanks go to Daniel Apsley and Gari Chua for many long, exciting discussions surrounding this work. Finally, it is a pleasure to thank Purnaprajna Bangere, Harold Blum, Paolo Cascini, Seungsu Lee, Mircea Mustață, Dhruv Sidana, Burt Totaro and Joe Waldron for their interest in this work. 

\subsection*{AI statement}
All new results in the main body of this paper, ideas behind them and their proofs are mine, worked out by hand. This includes interpolation construction of
Section~\ref{sec:affine}, Theorem~\ref{thm:main} and its proof,  homogeneous construction in Section~\ref{sec:projective} and the applications
in Section~\ref{sec:consequences}. I used Claude (Anthropic) to TeX and edit the paper, to check my proofs and to check every citation against its source. Its mathematical contributions are following simplifications and corrections:
\begin{enumerate}
\item It noticed that one line was missing in the step from strong
$F$-regularity of the pair $(\Spec R_c,\Delta_\psi)$ to that of the algebra
$R_c\langle\psi_c\rangle$. This line, using \eqref{eq:principal-comparison}
and \eqref{eq:tau-detects}, is now in the proof of
Lemma~\ref{lem:compose-theta}.
\item When I asked it to simplify Section~\ref{sec:projective}, it rewrote my
descent through Laurent charts in terms of Cartier maps
(Lemmas~\ref{lem:laurent-maps} and~\ref{lem:laurent-test-ideal},
\eqref{eq:psi-chart} and the gluing in Proposition~\ref{prop:Proj-descent}).

\item In my proof of Theorem~\ref{thm:main}, it filled in 2 details: why
\eqref{eq:generic-condition} follows from the definition of test modules, and
how to combine maps from the factors of a normal ring inside the algebra, in
the proof of Corollary~\ref{cor:normal-principal}.
\item It shortened a few proofs, for example those of Lemma~\ref{lem:iterate}
and Corollary~\ref{thm:intro-affine}. It suggested small changes for
precision: \eqref{eq:source-rule}
and the remark $\phi\in\cC_e(R,\Delta_\phi)$ in
Section~\ref{sec:map-divisor}.
\item \Cref{example-AI-1} and \Cref{example-AI-2} were found by ChatGPT 6.1 Sol for which we take no credit. 
\end{enumerate}

\section{Background}\label{sec:background}

Nothing in this section is new. Here we fix notation and recall the facts we will use. Most of them are in \cite{Blickle2013,SchwedeFAdjunction2009,Schwede2011,SchwedeSmithBook2026}, see
also \cite{BlickleSchwede2013}.

\subsection{Cartier algebras}

Let $R$ be a normal $F$-finite domain,  $X=\Spec R$. The module $F_*^eR$ is
$R$ with scalar multiplication $r\cdot F_*^ex=F_*^e(r^{p^e}x)$.
$\cC_e(R)=\Hom_R(F_*^eR,R)$ is the module of $p^{-e}$-linear maps on $R$,
$\cC(R)=\bigoplus_{e\geq0}\cC_e(R)$. We follow the conventions of
\cite[Definition~2.1 and Remark~2.5]{Blickle2013}. There are $2$ scalar actions
on $\cC_e(R)$: the source action of $F_*^eR$, given by
$((F_*^er)\cdot\phi)(F_*^ex)=\phi(F_*^e(rx))$ and multiplication on the
target, $r\phi=(F_*^er^{p^e})\cdot\phi$. Define product of
$\phi\in\cC_d(R)$ and $\eta\in\cC_f(R)$ by
\begin{equation}\label{eq:cartier-product}
 \phi\st\eta:=\phi\circ F_*^d\eta\in\cC_{d+f}(R),
\end{equation}
as in \cite[Section~3]{Schwede2011}. Unwinding the definitions, for $u,v\in R$
we have
\begin{equation}\label{eq:source-rule}
 \bigl((F_*^du)\cdot\phi\bigr)\st\bigl((F_*^fv)\cdot\eta\bigr)
 =\bigl(F_*^{d+f}(u^{p^f}v)\bigr)\cdot(\phi\st\eta).
\end{equation}
We write $\psi^{[n]}$ for the $n$-fold product of a homogeneous map $\psi$
with itself. $\psi^{[0]}=\operatorname{id}_R$.

A Cartier subalgebra is a graded subring $\cD\subseteq\cC(R)$ with $\cD_0=R$,
i.e.\ a pair $(R,\cD)$ in the sense of \cite[Definition~3.4]{Schwede2011}. Its
homogeneous pieces are stable under the source action, as
$\phi\st r=(F_*^er)\cdot\phi$ for $r\in\cD_0$. For $s>0$ and
$0\neq\psi\in\cC_s(R)$, let $R\langle\psi\rangle$ be the Cartier subalgebra
generated by $R$ and $\psi$. It lives in degrees divisible by $s$. We have
\begin{equation}\label{eq:principal-piece}
 R\langle\psi\rangle_{ns}=(F_*^{ns}R)\cdot\psi^{[n]}\qquad(n\geq1)
\end{equation}
by \cite[Remark~4.3]{Schwede2011}.

\subsection{Big test ideals}

On a domain, a Cartier subalgebra $\cD$ is nondegenerate in the sense of
\cite[Definition~3.8]{Schwede2011} as soon as it contains a nonzero map of
positive degree. For such $\cD$, big test ideal $\tauB(R;\cD)$ is the
smallest ideal $J$ meeting $R^\circ$ with $\phi(F_*^eJ)\subseteq J$ for all
$e>0$ and $\phi\in\cD_e$
\cite[Definitions~3.15--3.16 and Theorem~3.18]{Schwede2011}.

\begin{definition}[{\cite[Definition~3.19]{Schwede2011}}]\label{def:hdbs}
Let $\cD$ be nondegenerate. An element $c\in R^\circ$ is a homogeneously
defined big sharp test element, or HDBS test element for short, for
$(R,\cD)$ if for every $d\in R^\circ$ there are $e>0$ and $\phi\in\cD_e$
with $c\in\phi(F_*^e(dR))$.
\end{definition}

HDBS test elements exist \cite[Proposition~3.21]{Schwede2011}. An HDBS test element times an element of $R^\circ$ is again HDBS test element
\cite[Remark~3.20]{Schwede2011}. For an HDBS test element $c$,
\begin{equation}\label{eq:hdbs-formula}
 \tauB(R;\cD)=\sum_{e\geq0}\ \sum_{\phi\in\cD_e}\phi(F_*^e(cR))
\end{equation}
by the proof of \cite[Theorem~3.18]{Schwede2011}. 
Following \cite[Definitions~3.11 and~3.12]{Schwede2011}, $(R,\cD)$ is sharply
$F$-pure if some $\phi\in\cD_e$ with $e>0$ is surjective. It is strongly
$F$-regular if for every $d\in R^\circ$ there are $e>0$ and $\phi\in\cD_e$ with
$1\in\phi(F_*^e(dR))$. For nondegenerate $\cD$, $\cD'$ and a multiplicative
set $W$ we use \cite[Proposition~3.23]{Schwede2011}:
\begin{align}
 \cD'\subseteq\cD
 &\Longrightarrow\tauB(R;\cD')\subseteq\tauB(R;\cD),\label{eq:monotonicity}\\
 W^{-1}\tauB(R;\cD)
 &=\tauB(W^{-1}R;W^{-1}\cD),\label{eq:localization}\\
 \tauB(R;\cD)=R
 &\Longleftrightarrow (R,\cD)\text{ is strongly $F$-regular}.\label{eq:tau-detects}
\end{align}

\subsection{Map--divisor correspondence}\label{sec:map-divisor}

For an integral Weil divisor $D$ on $X$, $R(D)$ is its divisorial module, i.e. the
fractional ideal of rational functions $f$ with $\operatorname{div}(f)+D\geq0$
\cite[Tag~\stacktag{0EBM}]{stacks-project}. Let $K_X$ be the canonical divisor, $e>0$ and $q=p^e$. Write $\omega_R=R(K_X)$. Frobenius compatibility gives
$\Hom_R(F_*^eR,\omega_R)\cong F_*^e\omega_R$
\cite[Corollary~1.0.3]{BhattBlickleSchwedeTucker2026}.
On the regular locus, $\omega_R$ is invertible. So tensoring with
$\omega_R^{-1}$ and using the projection formula gives
\[
 \begin{aligned}
 \Hom_R(F_*^eR,R)
 &\cong (F_*^e\omega_R)\otimes_R\omega_R^{-1}\\
 &\cong F_*^eR((1-q)K_X).
 \end{aligned}
\]
First and last modules are reflexive. So their isomorphism extends
to all of $X$ \cite[Proposition~2.4(2)--(4)]{SchwedeSmith2010}. Here $F_*^eR$ acts on the left \cite[Section~4, Equation~(10)]{BlickleSchwede2013}. So a nonzero
$\phi\in\cC_e(R)$ gives an effective Weil divisor $D_\phi\sim(1-q)K_X$, and
$\Delta_\phi=D_\phi/(q-1)$ satisfies
\begin{equation}\label{eq:log-trivial}
 \Delta_\phi\geq0,\qquad(q-1)(K_X+\Delta_\phi)\sim0
\end{equation}
\cite[Theorem~3.11]{SchwedeFAdjunction2009}. Locally: if $\Phi$ generates
$\Hom_{R_P}(F_*^eR_P,R_P)$ at a height $1$ prime $P$ and
$\phi=(F_*^eu)\cdot\Phi$ there, then the coefficient of $D_\phi$ along $P$ is
the valuation of $u$ (see proof of
\cite[Theorem~3.11]{SchwedeFAdjunction2009}). In particular, $D_\phi$ does not
depend on the choice of $K_X$. Note $D_\theta=0$ if $\theta$ generates
$\cC_e(R)$ as $F_*^eR$-module.

For a $\Q$-divisor $\Gamma\geq0$, let $\cC_e(R,\Gamma)\subseteq\cC_e(R)$ be
the image of the restriction map from $\Hom_R(F_*^eR(\ceil{(p^e-1)\Gamma}),R)$.
These form a Cartier subalgebra $\cC(R,\Gamma)$. We have
$\tauB(R,\Gamma)=\tauB(R;\cC(R,\Gamma))$
\cite[Remarks~3.10 and~3.17]{Schwede2011}. We call $(\Spec R,\Gamma)$ strongly
$F$-regular, resp.\ sharply $F$-pure, if $(R,\cC(R,\Gamma))$ is so
\cite{HochsterHuneke1989,HaraWatanabe2002,Schwede2008}, also see \cite[Definition~3.1 and Remark~3.2]{SchwedeSmith2010}.
Algebra of one map and the pair it defines have the same test ideal
\begin{equation}\label{eq:principal-comparison}
 \tauB\bigl(R;R\langle\phi\rangle\bigr)=\tauB(R,\Delta_\phi)
\end{equation}
\cite[Remark~4.3 and Corollary~4.8]{Schwede2011}. Moreover, if
$\phi\in\cC_e(R)$ is surjective, then $(\Spec R,\Delta_\phi)$ is sharply
$F$-pure, since $\phi$ itself lies in $\cC_e(R,\Delta_\phi)$
\cite[Theorem~3.11(c),(d)]{SchwedeFAdjunction2009}.

We need $2$ more facts.

\begin{lemma}[{\cite[Exercise~4.12]{BlickleSchwede2013}}]\label{lem:product-divisor}
Let $d,f>0$, $0\neq\phi\in\cC_d(R)$ and $0\neq\eta\in\cC_f(R)$. Then
$\phi\st\eta\neq0$. With the product convention
\eqref{eq:cartier-product} we have $D_{\phi\st\eta}=p^fD_\phi+D_\eta$.
Consequently,
\begin{equation}\label{eq:product-Delta}
 \Delta_{\phi\st\eta}
 =\frac{p^f(p^d-1)}{p^{d+f}-1}\Delta_\phi
  +\frac{p^f-1}{p^{d+f}-1}\Delta_\eta.
\end{equation}
\end{lemma}

Localize at a height $1$ prime $P$. Choose a generator $T$ of $\Hom_{R_P}(F_*R_P,R_P)$.
Its iterates generate in every degree
\cite[Lemma~3.9]{SchwedeFAdjunction2009}. So write
$\phi=(F_*^du)\cdot T^{[d]}$ and
$\eta=(F_*^fv)\cdot T^{[f]}$ there.
\eqref{eq:source-rule} gives the source coefficient
$u^{p^f}v$. Hence $D_{\phi\st\eta}=p^fD_\phi+D_\eta$; compare the proof of
\cite[Lemma~4.9(i)]{SchwedeSmith2010}.
 Composition is nonzero because $\eta$ is surjective over $\Frac R$. The coefficients in
\eqref{eq:product-Delta} add up to $1$. So $\Delta_{\phi\st\eta}$ is a weighted average of $\Delta_\phi$ and $\Delta_\eta$. 

Second fact lets us go from sharp $F$-purity to strong $F$-regularity by making the boundary smaller.

\begin{lemma}[{\cite[Lemma~4.9(iii)]{SchwedeSmith2010}}]\label{lem:SS-convexity}
If $(\Spec R,\Gamma)$ is strongly $F$-regular and $(\Spec R,\Gamma+E)$ is
sharply $F$-pure for some $E\geq0$, then $(\Spec R,\Gamma+\lambda E)$ is
strongly $F$-regular for every rational $0<\lambda<1$.
\end{lemma}

\subsection{Test modules}\label{sec:test-modules}

In this subsection $R$ is a reduced $F$-finite ring, and
$\mathscr B=\bigoplus_{e\geq0}\mathscr B_e$ is a Cartier algebra over $R$ acting on an ideal $I\subseteq R$
\cite[Chapter~5, Definitions~8.1 and~8.4]{SchwedeSmithBook2026}, see also
\cite[Definition~2.2]{Blickle2013}. Here $\mathscr B_0=R$ acts by multiplication and $\phi\in\mathscr B_e$ acts by an $R$-linear map $F_*^eI\to I$,
$F_*^ex\mapsto\phi(F_*^ex)$. This action need not extend to $R$.
Products, iterates, source multiples and the subalgebra $R\langle\phi\rangle$
generated by $R$ and $\phi$ are as in Section~\ref{sec:background}. Test module $\tau(I,\mathscr B)$ is the smallest
$\mathscr B$-submodule $N\subseteq I$ such that $H^0_Q(N)_Q\to H^0_Q(I_Q)$ is a
nil-isomorphism for every associated prime $Q$ of $I$
\cite[Chapter~5, Definition~8.17]{SchwedeSmithBook2026}, see also
\cite{BlickleStabler2019}. It exists by
\cite[Chapter~5, Theorem~8.28]{SchwedeSmithBook2026},
\cite{Blickle2013}.

As $R$ is reduced, every associated prime $Q$ of $I$ is a minimal prime of
$R$. So $R_Q$ is a field, $H^0_Q(N)_Q=N_Q$ and $I_Q=R_Q$. In particular $N_Q$
is $0$ or $I_Q$. Let $\mathscr B_+=\bigoplus_{e>0}\mathscr B_e$. If some homogeneous map in $\mathscr B_+$ is nonzero on $I_Q$, it is surjective, as a nonzero
$p^{-e}$-linear map on a field. So are its iterates. Then no power of
$\mathscr B_+$ kills $I_Q$. So $N_Q\to I_Q$ is a nil-isomorphism iff
$N_Q=I_Q$. Otherwise $\mathscr B_+$ kills $I_Q$ and $N_Q\to I_Q$ is always a
nil-isomorphism. So, with
\[
 \Sigma(\mathscr B,I)=\{Q\in\operatorname{Min}R:\ \mathscr B_+\text{ acts
 nontrivially on }I_Q\},
\]
$\tau(I,\mathscr B)$ is the smallest $\mathscr B$-submodule $N\subseteq I$ with
\begin{equation}\label{eq:generic-condition}
 N_Q=I_Q\qquad\text{for all }Q\in\Sigma(\mathscr B,I).
\end{equation}
Also,
\begin{equation}\label{eq:generic-vanish}
 \tau(I,\mathscr B)_Q=0\qquad\text{for }Q\in\operatorname{Min}R\setminus
 \Sigma(\mathscr B,I).
\end{equation}
The elements of $I$ which vanish in $R_Q$ for all these $Q$ form a
$\mathscr B$-submodule $N_0$. For $P\in\Sigma(\mathscr B,I)$ pick $g$ in all
other minimal primes but not in $P$. Then $gI\subseteq N_0$ and $(gI)_P=I_P$.
So $N_0$ satisfies \eqref{eq:generic-condition}.
$\tau(I,\mathscr B)\subseteq N_0$. If $\mathscr B'\subseteq\mathscr B$ is a
Cartier subalgebra, then $\Sigma(\mathscr B',I)\subseteq\Sigma(\mathscr B,I)$, and
\eqref{eq:generic-condition} gives
\begin{equation}\label{eq:module-monotone}
 \tau(I,\mathscr B')\subseteq\tau(I,\mathscr B)
\end{equation}
\cite[Chapter~5, Exercise~8.1]{SchwedeSmithBook2026}.

For a normal domain $R$, $I=R$ and $\mathscr B\subseteq\cC(R)$ with
$\mathscr B_e\neq0$ for some $e>0$, condition \eqref{eq:generic-condition}
just says $N\neq0$. So $\tau(R,\mathscr B)=\tauB(R;\mathscr B)$. A normal $F$-finite ring $S$ is a
finite product of normal domains $S_i$
\cite[Tag~\stacktag{0357}]{stacks-project}. Maps in $\cC(S)$ preserve
the factors, as idempotents $\epsilon$ satisfy $\epsilon^{p^e}=\epsilon$. If a
Cartier subalgebra $\cE\subseteq\cC(S)$ has a nonzero map of positive degree
on every factor, we put $\tauB(S;\cE)=\prod_i\tauB(S_i;\cE_i)$, where $\cE_i$
is the image of $\cE$ in $\cC(S_i)$. This is the smallest $\cE$-stable ideal
of $S$ which is nonzero on every factor.

\section{Normal domains}\label{sec:affine}

Throughout this section $R$ is a normal $F$-finite domain and
$\cE\subseteq\cC(R)$ is a Cartier subalgebra (we keep the letter $D$ for
degrees). We follow Section~\ref{sec:idea}: first we prove principalization
when $\cE_1\neq0$ (Proposition~\ref{prop:principal-degree-one}). Then we
remove this assumption (Corollary~\ref{cor:normal-principal}).

\subsection{Test element for one map}

To show that a value of $\psi$ lies in $\tauB(R;R\langle\psi\rangle)$, we need
a test element for $R\langle\psi\rangle$ which we can control. Next lemma is
\cite[Lemma~5.7]{Schwede2011} with a weaker hypothesis: there $R_c$ is regular
and $\psi$ generates $\Hom_{R_c}(F_*^sR_c,R_c)$. Here we only assume that
$(R_c,R_c\langle\psi_c\rangle)$ is strongly $F$-regular. We need this, because
our final map need not generate after localization.

\begin{lemma}[{cf.~\cite[Proposition~6.14]{SchwedeFAdjunction2009} and \cite[Lemma~5.7]{Schwede2011}}]\label{lem:principal-test-element}
Let $0\neq\psi\in\cE_s$ with $s>0$ and put $\cA=R\langle\psi\rangle$.
Suppose that $c\in R^\circ$, that $(R_c,\cA_c)$ is strongly $F$-regular,
and that
\begin{equation}\label{eq:principal-fixed-q}
 c\in\psi(F_*^s(cR)).
\end{equation}
Then $c^2$ is an HDBS test element for $(R,\cA)$. We have
\begin{equation}\label{eq:principal-tau-formula-q}
 \tauB(R;\cA)=\sum_{t\geq0}\psi^{[t]}(F_*^{ts}(c^2R)).
\end{equation}
\end{lemma}

\begin{proof}
\setcounter{claim}{0}
Fix $d\in R^\circ$. We have to find $n\geq1$ with
$c^2\in\psi^{[n]}(F_*^{ns}(dR))$. We follow the idea of \cite[Lemma~5.7]{Schwede2011}. After inverting $c$ we reach a power of $c$ from $d$.
\eqref{eq:principal-fixed-q} brings it down to $c^2$. Only Claim~\ref{cl:power-of-c} uses that $(R_c,\cA_c)$ is strongly
$F$-regular.

\begin{claim}\label{cl:power-of-c}
There are $n\geq1$ and $k\geq0$ with $c^k\in\psi^{[n]}(F_*^{ns}(dR))$.
\end{claim}

\begin{proof}[Proof of Claim~\ref{cl:power-of-c}]
By \eqref{eq:principal-piece}, maps of positive degree in $\cA_c$ are source
multiples of iterates of $\psi_c$. So strong $F$-regularity of $(R_c,\cA_c)$
gives $n\geq1$ and $y\in R_c$ with $\psi_c^{[n]}(F_*^{ns}(dy))=1$. Write
$y=z/c^j$ with $z\in R$. Take $k\geq0$ with $kp^{ns}\geq j$. Then
\[
 \psi^{[n]}\bigl(F_*^{ns}(dz\,c^{kp^{ns}-j})\bigr)
 =c^k\,\psi_c^{[n]}\bigl(F_*^{ns}(dy)\bigr)=c^k
\]
in $R_c$, hence in $R$.
\end{proof}

\begin{claim}\label{cl:c-squared}
Let $l\geq1$. Then $c\in\psi^{[l]}(F_*^{ls}(cR))$, hence
$c^2\in\psi^{[l]}(F_*^{ls}(c^{p^{ls}+1}R))$.
\end{claim}

\begin{proof}[Proof of Claim~\ref{cl:c-squared}]
For $l=1$, first assertion is \eqref{eq:principal-fixed-q}. If
$c=\psi^{[l]}(F_*^{ls}(cx))$ with $x\in R$, then $cx$ lies in the ideal
$\psi(F_*^s(cR))$ by \eqref{eq:principal-fixed-q}. So
$c\in\psi^{[l]}(F_*^{ls}\psi(F_*^s(cR)))=\psi^{[l+1]}(F_*^{(l+1)s}(cR))$.
For second assertion, multiply $c=\psi^{[l]}(F_*^{ls}(cx))$ by $c$, which
gives $c^2=\psi^{[l]}(F_*^{ls}(c^{p^{ls}+1}x))$.
\end{proof}

Now we put the claims together. Take $n,k$ as in Claim~\ref{cl:power-of-c}
and $l\geq1$ with $p^{ls}+1\geq k$. Then Claim~\ref{cl:c-squared} gives
$c^2\in\psi^{[l]}(F_*^{ls}(c^kR))$. As $\psi^{[n]}(F_*^{ns}(dR))$ is an ideal
containing $c^k$, we get
\[
 c^2\in\psi^{[l]}\bigl(F_*^{ls}(c^kR)\bigr)
    \subseteq\psi^{[l]}\bigl(F_*^{ls}\psi^{[n]}(F_*^{ns}(dR))\bigr)
    =\psi^{[l+n]}\bigl(F_*^{(l+n)s}(dR)\bigr).
\]
So $c^2$ is HDBS for $(R,\cA)$. Finally, \eqref{eq:hdbs-formula} for $c^2$ is
\eqref{eq:principal-tau-formula-q} as maps in
$\cA$ are source multiples of the $\psi^{[t]}$. Source multiples do not
make $\psi^{[t]}(F_*^{ts}(c^2R))$ bigger.
\end{proof}

\subsection{Common test element}

Next we fix a map $\theta$ of degree $1$ and a test element $c$ with
$c\in\theta(F_*(cR))$. We include this well-known proof as Section~\ref{sec:projective} needs a graded version.

\begin{lemma}[{cf.~\cite[Section~6]{HochsterHuneke1990} and proof of \cite[Proposition~5.6]{Schwede2011}}]\label{lem:common-data}
Let $0\neq\theta\in\cE_1$. There is $c\in R^\circ$ such that $R_c$ is
regular, $c$ is HDBS for $(R,\cE)$ and $c=\theta(F_*(ch))$ for some $h\in R$. Moreover, for every $e\geq1$,
\begin{equation}\label{eq:full-after-localization}
 (\cE_e)_c=\Hom_{R_c}(F_*^eR_c,R_c)=(F_*^eR_c)\cdot\theta_c^{[e]}.
\end{equation}
Every nonzero multiple of $c$ has these properties, with a possibly
different $h$.
\end{lemma}

\begin{proof}
\setcounter{claim}{0}
We first find an open set where $R$ is regular and $\theta$ generates. Then we turn a function which vanishes outside this set into $c$.

\begin{claim}\label{cl:good-locus}
There is $b\in R^\circ$ such that $R_b$ is regular and
\begin{equation}\label{eq:theta-generates-b}
 \Hom_{R_b}(F_*R_b,R_b)=(F_*R_b)\cdot\theta_b.
\end{equation}
\end{claim}

\begin{proof}[Proof of Claim~\ref{cl:good-locus}]
Both conditions hold at the generic point. And both conditions are open.
For generation, the cokernel of $F_*R\to\Hom_R(F_*R,R)$,
$x\mapsto x\cdot\theta$, is a finite module. It vanishes at the generic
point, as Hom out of a finitely presented module commutes with localization
\cite[Tag~\stacktag{0583}]{stacks-project}. Regular locus of $R$ is open. So we can take $b\neq0$ with $D(b)$ inside both open sets.
\end{proof}

\begin{claim}\label{cl:fixed-element}
For $b$ as in Claim~\ref{cl:good-locus} there are $k\geq1$ and $h_0\in R$
such that $c_0=b^k$ satisfies $\theta(F_*(c_0h_0))=c_0$.
\end{claim}

\begin{proof}[Proof of Claim~\ref{cl:fixed-element}]
As $R_b$ is regular, $F_*R_b$ is faithfully flat \cite{Kunz1969}. So $R_b\to F_*R_b$ is pure. It
splits because its cokernel is finitely presented
\cite[Tag~\stacktag{058L}]{stacks-project}. By \eqref{eq:theta-generates-b}
the splitting is a source multiple of $\theta_b$. So $\theta_b(F_*u)=1$ for
some $u=y/b^N$ with $y\in R$. Now we clear the denominator with a $p$th
power. Choose $k\geq1$ with $k(p-1)\geq N$ and put $h_0=b^{k(p-1)-N}y$.
Then $c_0h_0=b^{kp}u$. So $\theta(F_*(c_0h_0))=b^k\theta_b(F_*u)=c_0$.
\end{proof}

Next we check that $c=c_0^2$ and $h=c_0^{p-1}h_0$ works. First,
$\theta(F_*(ch))=c_0\,\theta(F_*(c_0h_0))=c$. Next, $R_c=R_b$ is regular. The iterates $\theta_c^{[e]}$ generate
$\Hom_{R_c}(F_*^eR_c,R_c)$ by \eqref{eq:theta-generates-b} and
\cite[Lemma~3.9]{SchwedeFAdjunction2009}. As $\cE_e$ contains their source
multiples, this gives \eqref{eq:full-after-localization}. Finally, $\theta$
generates over $R_{c_0}=R_b$ and $c_0\in\theta(F_*(c_0R))$ by
Claim~\ref{cl:fixed-element}. So $c=c_0^2$ is HDBS for $R\langle\theta\rangle$
by \cite[Lemma~5.7]{Schwede2011}, hence also for $\cE$.

For a nonzero multiple, $gc$ is again HDBS
\cite[Remark~3.20]{Schwede2011}. $R_{gc}$ is a localization of $R_c$. It is therefore regular. \eqref{eq:full-after-localization} holds for $gc$. Finally,
$\theta(F_*((gc)g^{p-1}h))=g\,\theta(F_*(ch))=gc$.
\end{proof}

Composing with a source multiple of $\theta$ raises the degree of a map
without changing its value on a power of $c$. So we can move finitely many
values into one degree. We do that next.

\begin{lemma}[{cf.~proof of \cite[Lemma~5.7]{Schwede2011}}]\label{lem:synchronization}
Let $\theta,c$ be as in Lemma~\ref{lem:common-data}, let $L\geq1$,
and let $a_1,\ldots,a_m\in\tauB(R;\cE)$. There are $D>0$ and
$\alpha_i\in\cE_D$ with $\alpha_i(F_*^Dc^L)=a_i$ for $1\leq i\leq m$.
\end{lemma}

\begin{proof}
Write $c=\theta(F_*(ch))$ and put
$\rho=(F_*(c^{(p-1)(L-1)}h))\cdot\theta\in\cE_1$. Then
\[
 \rho(F_*c^L)=\theta\bigl(F_*((c^{L-1})^pch)\bigr)=c^{L-1}\theta(F_*(ch))=c^L.
\]
By induction $\rho^{[t]}(F_*^tc^L)=c^L$ for all $t\geq1$. So composing with
$\rho$ does not change values on $c^L$. Since $c^L$ is HDBS
\cite[Remark~3.20]{Schwede2011}, \eqref{eq:hdbs-formula} lets us write
$a_i=\sum_j\phi_{ij}(F_*^{d_{ij}}c^L)$ with $\phi_{ij}\in\cE_{d_{ij}}$ (we move
source coefficients into the maps). For $D>\max_{i,j}d_{ij}$ all these maps
$\alpha_i=\sum_j\phi_{ij}\st\rho^{[D-d_{ij}]}\in\cE_D$ have degree $D$. They satisfy
\[
 \alpha_i(F_*^Dc^L)
 =\sum_j\phi_{ij}\bigl(F_*^{d_{ij}}\rho^{[D-d_{ij}]}(F_*^{D-d_{ij}}c^L)\bigr)
 =\sum_j\phi_{ij}(F_*^{d_{ij}}c^L)=a_i.\qedhere
\]
\end{proof}

\subsection{Composing with \texorpdfstring{$\theta$}{theta}}

Composing with $\theta$
puts the values of a map on $F_*^{s_0}(c^2R)$ into the test ideal of a single
map.

\begin{lemma}[{cf.~proof of \cite[Theorem~4.3(i)]{SchwedeSmith2010}}]\label{lem:compose-theta}
Let $\theta,c$ be as in Lemma~\ref{lem:common-data}, let $s_0>0$, and
let $\psi_0\in\cE_{s_0}$ satisfy $c\in\psi_0(F_*^{s_0}(c^2R))$. Then
$\psi=\psi_0\st\theta\in\cE_{s_0+1}$ is nonzero,
$(R_c,R_c\langle\psi_c\rangle)$ is strongly $F$-regular, and
$\psi_0(F_*^{s_0}(c^2R))\subseteq\psi(F_*^{s_0+1}(c^2R))
\subseteq\tauB(R;R\langle\psi\rangle)$.
\end{lemma}

\begin{proof}
Write $c=\theta(F_*(ch))$. For $r\in R$ we have
$\theta(F_*(c^{p+1}hr^p))=cr\,\theta(F_*(ch))=c^2r$ and
$c^{p+1}hr^p\in c^2R$. Hence $c^2R\subseteq\theta(F_*(c^2R))$. This gives
\[
 \psi_0(F_*^{s_0}(c^2R))
 \subseteq\psi_0\bigl(F_*^{s_0}\theta(F_*(c^2R))\bigr)
 =\psi(F_*^{s_0+1}(c^2R)).
\]
In particular $c\in\psi(F_*^{s_0+1}(cR))$ and $\psi\neq0$.

For the second inclusion we apply Lemma~\ref{lem:principal-test-element} to
$\psi$. We just checked one of its hypotheses. The other one is that
$(R_c,R_c\langle\psi_c\rangle)$ is strongly $F$-regular. This is why we
composed with $\theta$. On $R_c$ the map $\psi_0$ is surjective. So $(\Spec R_c,\Delta_{\psi_0})$ is sharply $F$-pure. By \eqref{eq:full-after-localization},
$\theta_c$ generates $\Hom_{R_c}(F_*R_c,R_c)$. So $D_{\theta_c}=0$.
Lemma~\ref{lem:product-divisor} on $\Spec R_c$, with $d=s_0$ and $f=1$, gives
\[
 \Delta_\psi=\lambda\Delta_{\psi_0}\ \text{ on }\Spec R_c,\qquad
 \lambda=\frac{p(p^{s_0}-1)}{p^{s_0+1}-1}\in(0,1).
\]
 $R_c$ is regular and hence strongly $F$-regular
\cite[Theorem~3.1(c)]{HochsterHuneke1989}. So
by Lemma~\ref{lem:SS-convexity} $(\Spec R_c,\Delta_\psi)$ is strongly
$F$-regular. By \eqref{eq:principal-comparison} and \eqref{eq:tau-detects} $(R_c,R_c\langle\psi_c\rangle)$ is strongly $F$-regular too. Now
Lemma~\ref{lem:principal-test-element} says that $c^2$ is an HDBS test element
for $R\langle\psi\rangle$. \eqref{eq:principal-tau-formula-q} with $t=1$
gives $\psi(F_*^{s_0+1}(c^2R))\subseteq\tauB(R;R\langle\psi\rangle)$.
\end{proof}

\subsection{Prescribing finitely many values}

For interpolation we need many independent directions in $F_*^eR$. This is where we use $\dim R>0$. Following well-known lemma gives that.

\begin{lemma}\label{lem:nonperfect}
Let $R$ be a positive-dimensional normal domain of characteristic
$p>0$ with fraction field $K$. Then $K$ is not perfect. If $R$ is
$F$-finite, then $[K:K^{p^e}]=[K:K^p]^e$, which goes to infinity with
$e$.
\end{lemma}

\begin{proof}
The point is that a uniformizer of a DVR is not a $p$th power. Choose a
nonzero nonunit $u\in R$ and a prime $P$ minimal over $(u)$. By Krull's
principal ideal theorem $\operatorname{ht}P=1$. So $R_P$ is a DVR
\cite[Tags~\stacktag{00KV} and~\stacktag{00PD}]{stacks-project}. If $v$ is its
normalized valuation and $\pi$ a uniformizer, then $\pi=z^p$ in $K$ would
give $1=pv(z)$, which is impossible. So $[K:K^p]>1$. It is finite if $R$
is $F$-finite. Frobenius identifies $K/K^p$ with $K^{p^i}/K^{p^{i+1}}$.
\end{proof}

Now we prescribe finitely many values.

\begin{proposition}\label{prop:finite-interpolation}
Assume that $\dim R>0$ and $\cE_1\neq0$. For any
$a_1,\ldots,a_m\in\tauB(R;\cE)$ there are $s>0$ and $0\neq\psi\in\cE_s$
such that
\begin{equation}\label{eq:interpolation-containment}
 (a_1,\ldots,a_m)\subseteq\tauB(R;R\langle\psi\rangle)
                   \subseteq\tauB(R;\cE).
\end{equation}
\end{proposition}

\begin{proof}
\setcounter{claim}{0}
Choose $0\neq\theta\in\cE_1$ and $c$ as in Lemma~\ref{lem:common-data}. We
first build a dual basis in some $F_*^eR$ (Claim~\ref{cl:dual-basis}). We then construct maps taking the values $a_i$ and add them up.

\begin{claim}\label{cl:dual-basis}
After replacing $c$ by a nonzero multiple, one can find $e>0$, $L\geq1$,
elements $w_0,\ldots,w_m\in F_*^e(c^2R)$ and maps
$\beta_0,\ldots,\beta_m\in\cE_e$ with $\beta_i(w_j)=\delta_{ij}c^L$ for all
$i,j$.
\end{claim}

\begin{proof}[Proof of Claim~\ref{cl:dual-basis}]
Put $K=\Frac R$. By Lemma~\ref{lem:nonperfect} there is $e>0$ with
$[K:K^{p^e}]\geq m+1$. This is rank of $F_*^eR$. By generic freeness
\cite[Tag~\stacktag{051S}]{stacks-project}, after replacing $c$ by a nonzero
multiple (by Lemma~\ref{lem:common-data}), $P=(F_*^eR)_c$ is a free
$R_c$-module. Choose $m+1$ elements
$v_0,\ldots,v_m$ of a basis of $P$ and let $v_i^*\in\Hom_{R_c}(P,R_c)$ be
their coordinate functions. So $v_i^*(v_j)=\delta_{ij}$. Now we clear
denominators on both sides. By \eqref{eq:full-after-localization},
$(\cE_e)_c=\Hom_{R_c}(P,R_c)$. So there are $A\geq0$ and $\beta_i\in\cE_e$
with $(\beta_i)_c=c^Av_i^*$. Write
$v_j=c^{-B_j}\cdot F_*^ez_j$ with $B_j\geq0$ and $z_j\in R$, where $c^{-B_j}$
acts through the $R_c$-module structure of $P$. Choose $B>\max_jB_j$ and put
\[
 w_j=c^Bv_j=F_*^e\bigl(c^{(B-B_j)p^e}z_j\bigr).
\]
As $(B-B_j)p^e\geq p^e\geq2$, all $w_j$ lie in $F_*^e(c^2R)$. In $R_c$ we have
$\beta_i(w_j)=c^Av_i^*(c^Bv_j)=\delta_{ij}c^{A+B}$. Both sides lie in
$R$. So $L=A+B$ works.
\end{proof}

From now on $c$ is the element of Claim~\ref{cl:dual-basis}. We add the extra
value $a_0=c$, which lies in $\tauB(R;\cE)$ by \eqref{eq:hdbs-formula}. We
need it for Lemma~\ref{lem:compose-theta}. By
Lemma~\ref{lem:synchronization}, applied to $a_0,\ldots,a_m$ and $c^L$, there
are $D>0$ and $\alpha_i\in\cE_D$ with $\alpha_i(F_*^Dc^L)=a_i$. Put
$\psi_0=\sum_{i=0}^m\alpha_i\st\beta_i\in\cE_{D+e}$. Viewing $F_*^{D+e}R$ as
$F_*^D(F_*^eR)$, Claim~\ref{cl:dual-basis} gives
\[
 \psi_0(F_*^Dw_j)=\sum_{i=0}^m\alpha_i\bigl(F_*^D(\beta_i(w_j))\bigr)
 =\alpha_j(F_*^Dc^L)=a_j.
\]
So $\beta_j$ picks up $w_j$, $\alpha_j$ turns it into $a_j$. As all
$w_j$ lie in $F_*^e(c^2R)$, every $a_j$ is a value of $\psi_0$ on
$F_*^{D+e}(c^2R)$. In particular $c\in\psi_0(F_*^{D+e}(c^2R))$. Now
Lemma~\ref{lem:compose-theta} with $s_0=D+e$ shows that
$\psi=\psi_0\st\theta\in\cE_s$, $s=D+e+1$, is nonzero. It satisfies
$(a_1,\ldots,a_m)\subseteq\tauB(R;R\langle\psi\rangle)$. Second inclusion
in \eqref{eq:interpolation-containment} follows from
$R\langle\psi\rangle\subseteq\cE$ and \eqref{eq:monotonicity}.
\end{proof}

\subsection{Principalization}

\begin{proposition}\label{prop:principal-degree-one}
Assume $\cE_1\neq0$. Then there are $s>0$ and $0\neq\psi\in\cE_s$ with
$\tauB(R;\cE)=\tauB(R;R\langle\psi\rangle)$.
\end{proposition}

\begin{proof}
If $\dim R>0$, apply Proposition~\ref{prop:finite-interpolation} to generators
$a_1,\ldots,a_m$ of $\tauB(R;\cE)$. Both inclusions in
\eqref{eq:interpolation-containment} become equalities. If $\dim R=0$, then
$R$ is a field $\tauB(R;\cE)=R$. Choose $0\neq\psi\in\cE_1$. A nonzero
linear functional $F_*R\to R$ is surjective.
Also, $dR=R$ for $0\neq d\in R$.
Hence $1\in\psi(F_*(dR))$. Thus $(R,R\langle\psi\rangle)$ is strongly
$F$-regular and its test ideal is $R$ by \eqref{eq:tau-detects}.
\end{proof}

To drop the assumption $\cE_1\neq0$ we pass to Veronese, using
$2$ facts. 

\begin{lemma}\label{lem:iterate}
Let $0\neq\phi\in\cE_d$ with $d>0$ and let $n\geq1$. Then
$\tauB(R;R\langle\phi^{[n]}\rangle)=\tauB(R;R\langle\phi\rangle)$.
\end{lemma}

\begin{proof}
By Lemma~\ref{lem:product-divisor}, $\phi^{[n]}\neq0$.
$\Delta_{\phi^{[n]}}=\Delta_\phi$ by
\cite[Theorem~3.11(e)]{SchwedeFAdjunction2009}. Now use
\eqref{eq:principal-comparison}.
\end{proof}

If $\cE$ has a nonzero map of positive degree, then
\begin{equation}\label{eq:sum-principal}
 \tauB(R;\cE)=\sum_{d>0}\ \sum_{0\neq\phi\in\cE_d}\tauB(R;R\langle\phi\rangle)
\end{equation}
by the proof of \cite[Theorem~5.1]{Schwede2011}, see also
\cite[Chapter~5, Theorem~8.31]{SchwedeSmithBook2026}.

\begin{lemma}\label{lem:veronese}
Assume $\cE_d\neq0$ for some $d>0$ and let $n>0$. Then for
$\cE^{(n)}=\bigoplus_{e\geq0}\cE_{ne}$, we have $\tauB(R;\cE^{(n)})=\tauB(R;\cE)$.
\end{lemma}

\begin{proof}
For every $d>0$ and $0\neq\phi\in\cE_d$ we have $0\neq\phi^{[n]}\in\cE_{nd}$. So
Lemma~\ref{lem:iterate} and \eqref{eq:monotonicity} give
$\tauB(R;R\langle\phi\rangle)=\tauB(R;R\langle\phi^{[n]}\rangle)
\subseteq\tauB(R;\cE^{(n)})$. Summing up, \eqref{eq:sum-principal} gives
$\subseteq$. \eqref{eq:monotonicity} gives $\supseteq$.
\end{proof}

Now we can drop $\cE_1\neq0$. We state this for normal rings, as Section~\ref{sec:general} needs it on a
normalization.

\begin{corollary}\label{cor:normal-principal}
Let $S$ be a normal $F$-finite ring and let $\cE\subseteq\cC(S)$ be a Cartier
subalgebra with a nonzero map of positive degree on every factor of $S$.
Then there are $s>0$ and $\theta\in\cE_s$, nonzero on every factor, with
$\tauB(S;\cE)=\tauB(S;S\langle\theta\rangle)$.
\end{corollary}

\begin{proof}
First let $S$ be a domain. Choose $n>0$ with $\cE_n\neq0$. All maps in
$\cE^{(n)}$ have degree divisible by $n$. So we can run this section for
$\cE^{(n)}$ with Frobenius replaced by $F^n$. Read $F_*^e$ as $F_*^{ne}$, $p$
as $p^n$ and $\cE_e$ as $\cE_{ne}$. Use a nonzero map in $\cE_n$ as
$\theta$ in Lemma~\ref{lem:common-data}. Everything goes through. So Proposition~\ref{prop:principal-degree-one}
gives $0\neq\theta\in\cE_{ns}$ with
$\tauB(S;S\langle\theta\rangle)=\tauB(S;\cE^{(n)})$. This is $\tauB(S;\cE)$
by Lemma~\ref{lem:veronese}.

In general write $S=\prod_iS_i$ with normal domains $S_i$, idempotents
$\epsilon_i$ and images $\cE_i\subseteq\cC(S_i)$ of $\cE$. The domain case
gives $0\neq\theta_i\in(\cE_i)_{s_i}$ with
$\tauB(S_i;\cE_i)=\tauB(S_i;S_i\langle\theta_i\rangle)$. Put
$s=\operatorname{lcm}_i(s_i)$. By Lemma~\ref{lem:iterate},
$\theta_i'=\theta_i^{[s/s_i]}\in(\cE_i)_s$ has the same test ideal. Lift
$\theta_i'$ to some $\vartheta_i\in\cE_s$ and put
$\theta=\sum_i(F_*^s\epsilon_i)\cdot\vartheta_i$, which lies in $\cE_s$ as
$\cE_s$ is stable under source multiplication. Its image in $\cC(S_i)$ is
$\theta_i'$. So $S\langle\theta\rangle$ has image $S_i\langle\theta_i'\rangle$
in $\cC(S_i)$. Thus
$\tauB(S;S\langle\theta\rangle)=\prod_i\tauB(S_i;\cE_i)=\tauB(S;\cE)$.
\end{proof}

\begin{proof}[Proof of Corollary~\ref{thm:intro-principalization}]
Corollary~\ref{cor:normal-principal} for $S=R$ gives $s>0$ and
$0\neq\psi\in\cD_s$ with $\tauB(R;\cD)=\tauB(R;R\langle\psi\rangle)$. Second equality is \eqref{eq:principal-comparison}.
\eqref{eq:log-trivial} gives last assertion.
\end{proof}

\begin{proof}[Proof of Corollary~\ref{thm:intro-affine}]
We reduce to domains. Since $\Spec R$ is normal with finitely many
irreducible components, it is a disjoint union of normal integral schemes
\cite[Tag~\stacktag{0357}]{stacks-project}, i.e.\ $R=R_1\times\cdots\times R_t$
with all $R_i$ normal domains. Each $\cC_1(R_i)$ is nonzero. So Corollary~\ref{cor:normal-principal} for
$S=R$ and $\cE=\cC(R)$ gives $s>0$ and $\theta\in\cC_s(R)$, nonzero on every
factor, with $\tauB(R)=\tauB(R;R\langle\theta\rangle)$. Let $\Delta$ be
$\Delta_{\theta_i}$ on $\Spec R_i$, where $\theta_i$ is the image of $\theta$
in $\cC_s(R_i)$. Then \eqref{eq:principal-comparison} and
\eqref{eq:log-trivial} on each $R_i$ give $\tauB(R,\Delta)=\tauB(R)$ and
$(p^s-1)(K_X+\Delta)\sim0$.
\end{proof}

\section{Reduced rings}\label{sec:general}

In this section we prove Theorem~\ref{thm:main}. As explained in
Section~\ref{sec:idea}, we pass to the normalization of
$R/\operatorname{Ann}_R\tau(I,\mathscr B)$, apply Section~\ref{sec:affine}
there and then come back to $\mathscr B$. We use the notation of
Section~\ref{sec:test-modules}. $R$ is a reduced $F$-finite ring,
$\mathscr B$ acts on an ideal $I\subseteq R$, $J=\tau(I,\mathscr B)$ and
$\Sigma=\Sigma(\mathscr B,I)$. If $J=0$, then $s=1$ and $\psi=0\in\mathscr B_1$ work, as
$\tau(I,R\langle0\rangle)=0$ by \eqref{eq:generic-condition}. So from now on
$J\neq0$.

\subsection{Normalizing the test module}

By \eqref{eq:generic-condition} and \eqref{eq:generic-vanish}, $J_Q=I_Q=R_Q$
for $Q\in\Sigma$ and $J_Q=0$ for all other minimal primes $Q$ of $R$. As
$R\subseteq\prod_{Q\in\operatorname{Min}R}R_Q$, the annihilator
$H=\operatorname{Ann}_RJ$ is $\bigcap_{Q\in\Sigma}Q$. Thus $A=R/H$ is reduced,
its minimal primes are the $Q/H$ with $Q\in\Sigma$ and its total quotient
ring is $Q(A)=\prod_{Q\in\Sigma}R_Q$. Also $J\cap H=0$, as $x\in J\cap H$
gives $x^2\in HJ=0$. We view $J$ as an ideal of $A$. We have
$J\otimes_AQ(A)=Q(A)$. Hence every $\phi\in\mathscr B_e$ extends uniquely
from $J$ to a $p^{-e}$-linear map on $Q(A)$
\cite[Chapter~5, Observation~8.6]{SchwedeSmithBook2026}, which we also call $\phi$. It preserves the factors of $Q(A)$, as idempotents
satisfy $\epsilon^{p^e}=\epsilon$.

Let $S$ be the normalization of $A$ in $Q(A)$. As $A$ is $F$-finite, $S$ is finite over $A$
\cite[Tag~\stacktag{035S}]{stacks-project}. Thus $S$ is a normal $F$-finite
ring with one factor for each $Q\in\Sigma$. There is a nonzerodivisor
$d\in A$ with $dS\subseteq A$. Minimality forces $J$ to be an ideal of $S$. Following construction shows this.

\begin{lemma}(cf. \cite[Theorem~5.5(a)]{BlickleStabler2019}, \cite[Proposition~7.10]{SchwedeCenters2010})\label{lem:saturation}
For an $R$-submodule $N\subseteq J$, let $N^S=\{x\in N:Sx\subseteq N\}$ be the
largest ideal of $S$ contained in $N$.
\begin{enumerate}[label=(\alph*)]
\item If $N_Q=J_Q$ for all $Q\in\Sigma$, then $(N^S)_Q=J_Q$ for all
$Q\in\Sigma$.
\item If $\phi(F_*^eN)\subseteq N$ for some $\phi\in\mathscr B_e$, then
$\phi(F_*^eN^S)\subseteq N^S$.
\end{enumerate}
\end{lemma}

\begin{proof}
(a) For $x\in N$ we have $S(dx)=(dS)x\subseteq Ax\subseteq N$, as $N$ is an
$A$-module. So $dN\subseteq N^S$. Also $(dN)_Q=N_Q$ for $Q\in\Sigma$, as $d$ is
a unit in $R_Q$. (b) Let $x\in N^S$ and $u\in S$. In $Q(A)$ we have
$u\,\phi(F_*^ex)=\phi(F_*^e(u^{p^e}x))$. This lies in
$\phi(F_*^eN)\subseteq N$ because $u^{p^e}x\in N$.
\end{proof}

\begin{proposition}\label{prop:normalize-J}
\begin{enumerate}[label=(\alph*)]
\item\label{it:SJ}(cf. \cite[Theorem~5.1]{CarvajalRojasStabler2023}) $J$ is an ideal of $S$. In particular, $cS\subseteq J\subseteq A$ for $c\in J$.
\item\label{it:extend}(cf.\cite[Proposition~7.11]{SchwedeCenters2010}, \cite[Lemma~8.1]{SchwedeFAdjunction2009},
\cite[Exercise~6.14]{BlickleSchwede2013}) For $\phi\in\mathscr B_e$, the map $\phi$ on $Q(A)$ sends $S$ into $S$. It
 therefore restricts to $\overline\phi\in\cC_e(S)$.
\item\label{it:E} Let $\cE\subseteq\cC(S)$ be the Cartier subalgebra generated
by $S$ and all $\overline\phi$. Then $\tauB(S;\cE)=J$. For $e>0$, $\cE_e$
is the set of finite sums $\sum_i(F_*^eu_i)\cdot\overline{\phi_i}$ with
$\phi_i\in\mathscr B_e$ and $u_i\in S$.
\end{enumerate}
\end{proposition}

\begin{proof}
\ref{it:SJ}: By Lemma~\ref{lem:saturation}, $J^S$ is a $\mathscr B$-submodule
of $I$ which satisfies \eqref{eq:generic-condition}. So $J\subseteq J^S$ by
minimality, i.e.\ $J=J^S$ is an ideal of $S$.

\ref{it:extend}: We check this on each factor $S'$ of $S$ which is not a field. Let $V$ be the localization of $S'$ at a
height $1$ prime.
It is a DVR with uniformizer $t$. Write $JV=t^mV$, using
\ref{it:SJ}.  Localizing $\phi(F_*^eJ)\subseteq J$ gives
$\phi(F_*^eJV)\subseteq JV$. So for $z\in V$ we get
$\phi(F_*^ez)=t^{-m}\phi(F_*^e(t^{mp^e}z))\in t^{-m}JV=V$. And $S'$ is the
intersection of these $V$ \cite[Tag~\stacktag{031T}]{stacks-project}.

\ref{it:E}: The algebra $\cE$ preserves $J$, as the maps $\overline\phi$ do
and $J$ is an ideal of $S$. For $Q\in\Sigma$ some homogeneous $\phi\in\mathscr B_+$ is nonzero on $I_Q=J_Q$. So $\cE$ has a nonzero map of positive degree on every
factor of $S$. We have $\tauB(S;\cE)\subseteq J$, as $J$ is nonzero on every
factor. Conversely, $\tauB(S;\cE)\subseteq J\subseteq I$ is a
$\mathscr B$-submodule of $I$ which satisfies \eqref{eq:generic-condition}. Hence
$J\subseteq\tauB(S;\cE)$. Finally, by
uniqueness of extensions, $\overline{\phi\st\eta}=\overline\phi\st\overline\eta$.
So by \eqref{eq:source-rule} the sums, together with $\cE_0=S$,
are closed under products, where target multiplication by $w\in S$ is source
multiplication by $w^{p^e}$. Hence they form the algebra generated by $S$ and
the maps $\overline\phi$.
\end{proof}

\subsection{Proof of Theorem~\ref{thm:main}}

\begin{proof}[Proof of Theorem~\ref{thm:main}]
\setcounter{claim}{0}
By Proposition~\ref{prop:normalize-J}\ref{it:E} and
Corollary~\ref{cor:normal-principal} there are $e>0$ and $\theta\in\cE_e$,
nonzero on every factor of $S$, with
\begin{equation}\label{eq:theta-J}
 \tauB(S;S\langle\theta\rangle)=J.
\end{equation}
By Proposition~\ref{prop:normalize-J}\ref{it:E}, $\theta$ is a sum of maps
$\overline\phi$ with $\phi\in\mathscr B$ and with coefficients in $S$. These
coefficients need not come from $R$.
The plan: we choose $c\in J$, which will clear the coefficients
(Claim~\ref{cl:c-in-J}). We then take an iterate $\theta^{[N]}$ which maps
$F_*^{Ne}(c^4S)$ onto $J$ (Claim~\ref{cl:finite-iterate}). Then the source multiple $\eta=(F_*^{Ne}c^2)\cdot\theta^{[N]}$ still has test
ideal $J$
(Claim~\ref{cl:eta}). It comes from a map $\psi\in\mathscr B$
(Claim~\ref{cl:lift}). At the end we compare test modules of $\psi$ and
$\eta$.

\begin{claim}\label{cl:c-in-J}
There are $c\in J\cap S^\circ$ and $h\in S$ such that $S_c$ is regular,
$\theta_c$ generates $\Hom_{S_c}(F_*^eS_c,S_c)$ and $\theta(F_*^e(ch))=c$.
\end{claim}

\begin{proof}[Proof of Claim~\ref{cl:c-in-J}]
Let $S_i$ be the factors of $S$, $\epsilon_i$ the idempotents and
$0\neq\theta_i\in\cC_e(S_i)$ the image of $\theta$. Lemma~\ref{lem:common-data}
for $S_i\langle\theta_i\rangle$, with $F^e$ in place of $F$, gives $0\neq c_i\in S_i$,
$h_i\in S_i$ with $(S_i)_{c_i}$ regular, $\theta_i$ generating there by
\eqref{eq:full-after-localization} and $\theta_i(F_*^e(c_ih_i))=c_i$. By the
last assertion of that lemma, we can multiply $c_i$ by a nonzero element of
$\epsilon_iJ$. Adjust $h_i$ as in that lemma. Then $c=\sum_i\epsilon_ic_i$ and $h=\sum_i\epsilon_ih_i$ work,
because $J$ is an ideal of $S$.
\end{proof}

\begin{claim}\label{cl:finite-iterate}
There is $N\geq1$ with $\theta^{[N]}(F_*^{Ne}(c^4S))=J$.
\end{claim}

\begin{proof}[Proof of Claim~\ref{cl:finite-iterate}]
As in the proof of Lemma~\ref{lem:compose-theta}, for $r\in S$
we have $\theta(F_*^e(c^{3p^e+1}hr^{p^e}))=c^3r\,\theta(F_*^e(ch))=c^4r$. So
$c^4S\subseteq\theta(F_*^e(c^4S))$. The ideals
$\theta^{[n]}(F_*^{ne}(c^4S))$ increase with $n$. They lie in $J$ because
$c^4S\subseteq J$ by Proposition~\ref{prop:normalize-J}\ref{it:SJ} and $J$
is $\theta$-stable. Their union is a $\theta$-stable ideal which contains
$c^4\in S^\circ$. So it contains $J$ by \eqref{eq:theta-J}. As $S$ is
Noetherian, some $N$ works.
\end{proof}

\begin{claim}\label{cl:eta}
The map $\eta=(F_*^{Ne}c^2)\cdot\theta^{[N]}\in\cE_{Ne}$ satisfies
$\eta(F_*^{Ne}S)=\eta(F_*^{Ne}J)=J$ and $\tauB(S;S\langle\eta\rangle)=J$.
\end{claim}

\begin{proof}[Proof of Claim~\ref{cl:eta}]
By Claim~\ref{cl:finite-iterate},
$\eta(F_*^{Ne}(c^2S))=\theta^{[N]}(F_*^{Ne}(c^4S))=J$. As $c^2S\subseteq J$
and $J$ is $\theta$-stable,
\[
 J=\eta(F_*^{Ne}(c^2S))\subseteq\eta(F_*^{Ne}J)\subseteq\eta(F_*^{Ne}S)
 =\theta^{[N]}(F_*^{Ne}(c^2S))\subseteq J.
\]
In particular $c\in\eta(F_*^{Ne}(cS))$. Over $S_c$, $\eta$ is a source multiple of $\theta^{[N]}$ by unit $c^2$.
$\theta^{[N]}$ generates $\Hom_{S_c}(F_*^{Ne}S_c,S_c)$ by
Claim~\ref{cl:c-in-J} and \cite[Lemma~3.9]{SchwedeFAdjunction2009}. So
$(S_c,S_c\langle\eta_c\rangle)$ is strongly $F$-regular
\cite[Lemma~3.22]{Schwede2011}. Now apply Lemma~\ref{lem:principal-test-element},
on each factor of $S$. \eqref{eq:principal-tau-formula-q} gives
$\tauB(S;S\langle\eta\rangle)\supseteq\eta(F_*^{Ne}(c^2S))=J$. Other
inclusion holds as $J$ is $\eta$-stable and nonzero on every factor.
\end{proof}

\begin{claim}\label{cl:lift}
There is $\psi\in\mathscr B_{Ne}$ with $\overline\psi=\eta$.
\end{claim}

\begin{proof}[Proof of Claim~\ref{cl:lift}]
By Proposition~\ref{prop:normalize-J}\ref{it:E} we can write
$\theta^{[N]}=\sum_i(F_*^{Ne}u_i)\cdot\overline{\phi_i}$ with
$\phi_i\in\mathscr B_{Ne}$ and $u_i\in S$. Thus
$\eta=\sum_i(F_*^{Ne}(c^2u_i))\cdot\overline{\phi_i}$. Here $c^2u_i\in J$, as $J$ is an ideal of $S$ containing $c$
(Proposition~\ref{prop:normalize-J}\ref{it:SJ}). As $J\to A$ is injective,
$c^2u_i$ has a unique preimage $r_i\in J\subseteq R$. Put $\psi=\sum_i(F_*^{Ne}r_i)\cdot\phi_i\in\mathscr B_{Ne}$. For
$x\in J$ the image of $r_ix$ in $A$ is $c^2u_ix$. So
$\psi(F_*^{Ne}x)=\eta(F_*^{Ne}x)$. By uniqueness of extensions,
$\overline\psi=\eta$.
\end{proof}

Now we compare $\psi$ with $\eta$. Put $K=\tau(I,R\langle\psi\rangle)$. Then
$K\subseteq J$ by \eqref{eq:module-monotone}. For $Q\in\Sigma$, $\psi$ acts on
$I_Q=J_Q$ as $\eta$ does. This action is nonzero, as $\eta(F_*^{Ne}S)=J$ is
nonzero on every factor of $S$. So $\Sigma\subseteq\Sigma(R\langle\psi\rangle,I)$.
We have $K_Q=I_Q=J_Q$ for $Q\in\Sigma$ by \eqref{eq:generic-condition}. By Lemma~\ref{lem:saturation}, $K^S$ is an
ideal of $S$ which is nonzero on every factor. It is stable under $\psi$,
i.e.\ under $\eta=\overline\psi$. So $K^S$ contains
$\tauB(S;S\langle\eta\rangle)$, which is $J$ by Claim~\ref{cl:eta}. Thus
$J\subseteq K^S\subseteq K\subseteq J$. This gives $\tau(I,R\langle\psi\rangle)=J$.

It remains to compute the image of $\psi$. As $r_i\in J$, we have
$r_iI\subseteq J$. Every $\phi_i$ maps $J$ into $J$. So
$\psi(F_*^{Ne}I)\subseteq J$. And $\psi(F_*^{Ne}J)=\eta(F_*^{Ne}J)=J$ by
Claim~\ref{cl:eta}. So $\psi(F_*^{Ne}I)=\psi(F_*^{Ne}J)=J$ and $s=Ne$
works.
\end{proof}

\begin{remark}\label{rem:one-image}
For a normal $F$-finite domain $R$, $I=R$ and $\mathscr B=\cC(R)$,
the map $\psi$ of Theorem~\ref{thm:main} works in
Corollary~\ref{thm:intro-principalization}. Writing $J=\tauB(R)$, we have
\[
 \psi(F_*^sR)=\psi(F_*^sJ)=J.
\]
It follows by induction that
$\psi^{[n]}(F_*^{ns}R)=J$ for every $n\geq1$.
Non-$F$-pure ideal $\sigma(R,\Delta_\psi)$ is the stable image of
these iterates; cf.\
\cite[Definition~3.2 and the discussion following
Proposition~3.4]{PatakfalviSchwedeTucker2017}.
Hence $\sigma(R,\Delta_\psi)=J=\tauB(R)$.
In particular, for every prime $P$,
$(\Spec R_P,\Delta_\psi|_{\Spec R_P})$ is sharply $F$-pure if and only if
$R_P$ is strongly $F$-regular. Sharp $F$-purity is equivalent
to surjectivity of $\psi_P$
\cite[Theorem~3.11(d)]{SchwedeFAdjunction2009}. By \eqref{eq:localization}
\[
 \psi_P(F_*^sR_P)=J_P=\tauB(R_P),
\]
 and $\tauB(R_P)=R_P$ if and only if
$R_P$ is strongly $F$-regular.

We can also choose the boundary of Theorem~\ref{thm:intro-global} with this
property. Let $\psi\in\cC_s(S)$ and homogeneous $c$ be as in
Proposition~\ref{prop:homogeneous-interpolation}. Put
$J_\psi=\tauB(S;S\langle\psi\rangle)$. This agrees with
$\tauB(S)$ on $U$. So we run Claims~\ref{cl:finite-iterate}
and~\ref{cl:eta} with $J_\psi$ in place of $J$. By
Lemma~\ref{lem:compose-theta}, $(S_c,S_c\langle\psi_c\rangle)$ is strongly
$F$-regular and $c\in\psi(F_*^s(c^2S))\subseteq J_\psi$. So
$c=\psi(F_*^s(ch))$ for some $h\in S$. Then the computation in
Claim~\ref{cl:finite-iterate} gives $N\geq1$ with
$\psi^{[N]}(F_*^{Ns}(c^4S))=J_\psi$. As in Claim~\ref{cl:eta} the
homogeneous map $\eta=(F_*^{Ns}c^2)\cdot\psi^{[N]}$ satisfies
$\eta(F_*^{Ns}S)=\eta(F_*^{Ns}J_\psi)=J_\psi$. Over $S_c$, $\eta$ is a source
multiple of $\psi^{[N]}$ by a unit. So $(S_c,S_c\langle\eta_c\rangle)$ is
strongly $F$-regular by Lemma~\ref{lem:iterate} and \eqref{eq:tau-detects}. As $c\in J_\psi=\eta(F_*^{Ns}(c^2S))$, Lemma~\ref{lem:principal-test-element}
gives $\tauB(S;S\langle\eta\rangle)\supseteq J_\psi$. Equality holds as $J_\psi$ is
$\eta$-stable. So Proposition~\ref{prop:Proj-descent}
applies to $\eta$. On a chart $D_+(f)=\Spec B$ put $C=B[f,f^{-1}]$. By \eqref{eq:psi-chart},
$\eta_f$ is a source multiple by a unit of the Laurent extension of the induced map $\varphi_f$. This unit preserves every extended ideal $LC$. So \eqref{eq:Laurent-residue-map} gives
$\eta_f(F_*^{Ns}(LC))=\varphi_f(F_*^{Ns}L)C$ for every ideal $L\subseteq B$.
Also $(J_\psi)_f=\tauB(S)_f$ by Proposition~\ref{prop:homogeneous-interpolation}. This is $\tauB(B)C$ by \eqref{eq:localization} and Lemmas~\ref{lem:laurent-maps} and~\ref{lem:laurent-test-ideal}. So localizing the $2$ image equalities for
$\eta$ at $f$, with $L=B$ and $L=\tauB(B)$ and contracting to $B$, we get
$\varphi_f(F_*^{Ns}B)=\varphi_f(F_*^{Ns}\tauB(B))=\tauB(B)$.
Now the argument above works on every chart, with $\Delta=0$ on components of
dimension $0$.
\end{remark}

\section{Projective case}\label{sec:projective}

Let $S=\bigoplus_{n\geq0}S_n$ be a standard graded normal $F$-finite domain,
i.e.\ $S$ is generated by $S_1$ over $S_0$. Assume $S_1\neq0$. Put
$Y=\Proj S$ and $U=\Spec S\setminus V(S_+)$. We first build a homogeneous map
$\psi$ such that $\tauB(S;S\langle\psi\rangle)$ and $\tauB(S)$ agree on $U$.
Then we descend $\psi$ to $Y$. Interpolation is the same as in
Section~\ref{sec:affine}, we only have to keep track of degrees. To descend to
$Y$ we compare test ideals on Laurent extensions (Section~\ref{sec:laurent},
compare \cite[Application~5.2]{BravoSmith2002}).

Following \cite[\S5.3]{SchwedeSmith2010}, we grade $F_*^eS$ by
$\deg(F_*^es)=p^{-e}\deg s$. A map $\phi\in\cC_e(S)$ is homogeneous of degree
$\delta$ if it sends elements of degree $\gamma$ to elements of degree
$\gamma+\delta$. Since $F_*^eS$ is a finitely generated graded $S$-module, every
$\phi\in\cC_e(S)$ is a finite sum of homogeneous maps, its homogeneous
components. For homogeneous $\phi\in\cC_d(S)$ and $\eta\in\cC_f(S)$,
\begin{equation}\label{eq:graded-product-degree}
 \deg(\phi\st\eta)=\deg\phi+p^{-d}\deg\eta,
\end{equation}
because for homogeneous $s\in S$
\[
 \deg\bigl((\phi\st\eta)(F_*^{d+f}s)\bigr)
   =p^{-d}\bigl(p^{-f}\deg s+\deg\eta\bigr)+\deg\phi.
\]

We also need the charts of $Y$. Fix nonzero generators $x_1,\ldots,x_\ell$ of
$S_1$ over $S_0$. For $0\neq x\in S_1$, $D_+(x)$ is Spec of
$(S_x)_0=S_0[x_1/x,\ldots,x_\ell/x]$, which is $F$-finite. If $0\neq g\in S$ is homogeneous and divisible by $x$, then 
\begin{equation}\label{eq:laurent-chart}
 S_g=(S_g)_0[x,x^{-1}]
\end{equation}
(compare \cite[Chapter~3, Proposition~4.13(e)]{SchwedeSmithBook2026}). Also $(S_g)_0$ is normal. To
see this, suppose that
$u\in\Frac((S_g)_0)$ is integral over $(S_g)_0$. Then $u\in S_g$. $u$ has degree $0$ because $bu\in(S_g)_0$ for some $0\neq b\in(S_g)_0$. $(S_g)_0$ has $K(Y)$ as fraction field. Hence for
$S_g$ that is $K(Y)(x)$.

\subsection{Homogeneous maps with prescribed values}

First we make the data of Section~\ref{sec:affine} homogeneous.

\begin{lemma}\label{lem:homogeneous-common-data}
For $\cE=\cC(S)$, we can choose $\theta$ and $c$ in
Lemma~\ref{lem:common-data} homogeneous. By its last assertion, we may replace $c$ by any nonzero homogeneous multiple.
\end{lemma}

\begin{proof}
The proof of Lemma~\ref{lem:common-data} gives $c=b^{2k}$ for some $k\geq1$,
provided $S_b$ is regular and $\theta_b$ generates
$\Hom_{S_b}(F_*S_b,S_b)$. Thus it suffices to choose $\theta$ and $b$
homogeneous.

Since $S$ is $F$-finite domain, we can choose a nonzero homogeneous
$\theta\in\cC_1(S)$. The locus where $\theta$ fails to generate
$\cC_1(S)$ is cut out by the homogeneous ideal
$\operatorname{Ann}_S(\cC_1(S)/(F_*S)\cdot\theta)$.

The nonregular locus is also cut out by a homogeneous ideal. Put
$K=\Frac S$ and $r=[K:K^p]$. Since $S$ is a domain,
$\operatorname{Fitt}_{r-1}(F_*S)=0$. By \cite{Kunz1969} and
\cite[Tag~\stacktag{07ZD}]{stacks-project}, the nonregular locus is therefore
$V(\operatorname{Fitt}_r(F_*S))$. This Fitting ideal is homogeneous because
$F_*S$ has a graded presentation.

Neither locus contains the generic point. So both ideals are nonzero.
Choose a nonzero homogeneous $b$ in their product. Then $S_b$ is regular,
$\theta_b$ generates and $c=b^{2k}$ is homogeneous.
\end{proof}

Recall the following well-known lemma.

\begin{lemma}(\cite[Chapter~3, Proposition~4.2(b)]{SchwedeSmithBook2026})\label{lem:tau-homogeneous}
$\tauB(S)$ is homogeneous.
\end{lemma}

\begin{proposition}\label{prop:homogeneous-interpolation}
Assume $\dim Y>0$. Then there exist $s>0$ and a nonzero homogeneous
$\psi\in\cC_s(S)$ such that
$\tauB(S;S\langle\psi\rangle)\cO_U=\tauB(S)\cO_U$.
\end{proposition}

\begin{proof}
\setcounter{claim}{0}
We adapt the proof of Proposition~\ref{prop:finite-interpolation}, keeping
the maps homogeneous. For this, we need a homogeneous dual basis and
prescribed values of the same degree. Choose nonzero homogeneous generators
$b_1,\ldots,b_r$ of $I=\tauB(S)$
(Lemma~\ref{lem:tau-homogeneous}). Put $m=r\ell$. Choose homogeneous
$\theta,c$ as in Lemma~\ref{lem:homogeneous-common-data}.

\begin{claim}\label{cl:graded-dual-basis}
After replacing $c$ by a nonzero homogeneous multiple, we can find
$e>0$, $N=\deg c>\max_i\deg b_i$, integers $A\geq0$ and $B\geq1$,
homogeneous maps $\beta_0,\ldots,\beta_m\in\cC_e(S)$ of degree $AN$, and
homogeneous elements $w_0,\ldots,w_m\in F_*^e(c^2S)$ of degree $BN$ such that
\begin{equation}\label{eq:graded-delta}
 \beta_i(w_j)=\delta_{ij}c^L,\qquad L=A+B.
\end{equation}
\end{claim}

\begin{proof}[Proof of Claim~\ref{cl:graded-dual-basis}]
Replace $c$ by $cx_1$ and put $T=S_c$, $T_0=(S_c)_0$. Then $\deg c>0$.
\eqref{eq:laurent-chart} gives $T=T_0[x_1,x_1^{-1}]$. $T_0$
is a normal domain with fraction field $K(Y)$.

Since $\dim Y>0$, a codimension $1$ point of $Y$ has a DVR as its local
ring. Thus $K(Y)$ is not perfect by Lemma~\ref{lem:nonperfect}. Choose
$e>0$ such that $[K(Y):K(Y)^q]\geq m+1$. $F_*^eT_0$ has this rank as $T_0$-module. So generic freeness
\cite[Tag~\stacktag{051S}]{stacks-project} makes it free after inverting
some $0\neq g\in T_0$. Write $g=f/c^k$ with $f\in S$ homogeneous and
$k\geq0$. Replacing $c$ by $cf$ replaces $T_0$ by $(T_0)_g$.
We can now choose $v_0,\ldots,v_m$ from a basis of $F_*^eT_0$.

Choose $N>\max_i\deg b_i$ divisible by $\deg c$ and replace $c$ by
$c^{N/\deg c}$. This leaves $T$ unchanged and gives $\deg c=N$.

Let $v_i^*:F_*^eT_0\to T_0$ be the coordinate maps for those basis
vectors. Extend them to $F_*^eT\to T$ by
\[
 v_i^*(F_*^e(ax_1^n))=
 \begin{cases}
  x_1^{n/q}v_i^*(F_*^ea),&q\mid n,\\
  0,&q\nmid n,
 \end{cases}
 \qquad a\in T_0,\ n\in\Z.
\]
These extensions are $T$-linear and homogeneous of degree $0$. They still
satisfy $v_i^*(v_j)=\delta_{ij}$.

We now clear denominators. Since $F_*^eS$ is finitely presented, Hom
commutes with localization
\cite[Tag~\stacktag{0583}]{stacks-project}. Thus multiplying the targets
of all the $v_i^*$ by one power $c^A$, $A\geq0$, gives maps
$\beta_i\in\cC_e(S)$ with $(\beta_i)_c=c^Av_i^*$. Taking their homogeneous components of degree $AN$ preserves these localizations.

Write $v_j=c^{-B_j}\cdot F_*^ez_j$ with $B_j\geq0$ and $z_j\in S$
homogeneous. Choose $B>\max_jB_j$ and put
\[
 w_j=c^Bv_j=F_*^e(c^{q(B-B_j)}z_j).
\]
Then $\deg w_j=BN$. Also, $w_j\in F_*^e(c^2S)$ because $q(B-B_j)\geq2$.
Finally,
\[
 \beta_i(w_j)=c^Av_i^*(c^Bv_j)=\delta_{ij}c^{A+B}.
\]
This equality holds in $T$, hence in $S$, proving the claim.
\end{proof}

\begin{claim}\label{cl:one-degree}
$a_{ij}=b_ix_j^{N-\deg b_i}$ lie in $I=\tauB(S)$ and have degree $N$.
The ideal $I'=(a_{ij})$ satisfies
\begin{equation}\label{eq:equal-ideal-U}
 I'\cO_U=I\cO_U.
\end{equation}
\end{claim}

\begin{proof}[Proof of Claim~\ref{cl:one-degree}]
Each $a_{ij}$ is a multiple of $b_i$ of degree $N$. On $D(x_j)$, $x_j^{N-\deg b_i}$ is a unit. Hence $I'_{x_j}=I_{x_j}$. These open
sets cover $U$, giving \eqref{eq:equal-ideal-U}.
\end{proof}

It remains to put $I'$ into the test ideal of one homogeneous map.
Relabel $a_{ij}$ as $a_1,\ldots,a_m$ and set $a_0=c$. By
\eqref{eq:hdbs-formula}, $c\in I$. Now every $a_i$ has degree $N$.
Lemma~\ref{lem:synchronization} gives $D>0$ and maps
$\alpha_i\in\cC_D(S)$ such that
\[
 \alpha_i(F_*^Dc^L)=a_i.
\]
Keeping only the homogeneous component of $\alpha_i$ of degree
$N-p^{-D}LN$ preserves this equality.

Put $\psi_0=\sum_{i=0}^m\alpha_i\st\beta_i\in\cC_{D+e}(S)$. By
\eqref{eq:graded-product-degree}, each summand has degree
\[
 \deg(\alpha_i\st\beta_i)
 =N-p^{-D}LN+p^{-D}AN
 =N-p^{-D}BN.
\]
Thus $\psi_0$ is homogeneous. Moreover, \eqref{eq:graded-delta} gives
\[
 \psi_0(F_*^Dw_j)
 =\sum_i\alpha_i(F_*^D\beta_i(w_j))
 =\alpha_j(F_*^Dc^L)=a_j.
\]
Since $w_j\in F_*^e(c^2S)$, we have
\[
 c,a_1,\ldots,a_m\in\psi_0(F_*^{D+e}(c^2S)).
\]
Lemma~\ref{lem:compose-theta}, with $\cE=\cC(S)$ and $s_0=D+e$, therefore
gives a nonzero map
$\psi=\psi_0\st\theta\in\cC_s(S)$, where $s=D+e+1$, with
$I'\subseteq\tauB(S;S\langle\psi\rangle)$. $\psi$ is homogeneous
by \eqref{eq:graded-product-degree}. Monotonicity
\eqref{eq:monotonicity} gives
\[
 I'\subseteq\tauB(S;S\langle\psi\rangle)\subseteq I.
\]
Restricting to $U$ and using \eqref{eq:equal-ideal-U} proves the assertion.
\end{proof}

\begin{remark}\label{rem:a-positive}
    We may also arrange $\deg\psi>0$. The proof of
Lemma~\ref{lem:synchronization} allows any sufficiently large $D$,
while $N,B,e$ and $\theta$ remain fixed. Writing $\delta=\deg\theta$,
we have
\[
 \deg\psi
 =N-p^{-D}BN+p^{-(D+e)}\delta
 =N+p^{-D}(p^{-e}\delta-BN)>0
\]
for $D$ sufficiently large, since $N>0$.
\end{remark}

\subsection{Laurent extensions}\label{sec:laurent}

To descend from the cone to $Y$, we compare test ideals on $B$ and
$B[t,t^{-1}]$. For finitistic test ideal,
\cite[Theorem~4.1]{BravoSmith2002} showed
$\tau(A[t])=\tau(A)A[t]$, assuming test ideals commute with localization.
We prove the Laurent formula below for big test ideals of Cartier
algebras on arbitrary normal $F$-finite domains. Localization holds by
\eqref{eq:localization}. Compare also \cite[Application~5.2]{BravoSmith2002}.

Let $B$ be a normal $F$-finite domain and put $C=B[t,t^{-1}]$. For
$\alpha\in\cC_e(B)$ and $q=p^e$, define
$\widetilde\alpha\in\cC_e(C)$ by
\begin{equation}\label{eq:Laurent-residue-map}
 \widetilde\alpha(F_*^e(bt^n))=
 \begin{cases}
  \alpha(F_*^eb)\,t^{n/q},&q\mid n,\\
  0,&q\nmid n,
 \end{cases}
 \qquad(b\in B,\ n\in\Z).
\end{equation}
This is the smooth pullback action of
\cite[Theorem~5.5(c)]{BlickleStabler2019}, using $dt/t$ to trivialize
the relative canonical module.
The map is $C$-linear. Checking on monomials gives, for $b\in B$ and $\beta\in\cC_f(B)$,
\begin{equation}\label{eq:laurent-multiplicative}
 \widetilde{(F_*^eb)\cdot\alpha}
   =(F_*^eb)\cdot\widetilde\alpha,\qquad
 \widetilde{\alpha\st\beta}
   =\widetilde\alpha\st\widetilde\beta.
\end{equation}

\begin{lemma}\label{lem:laurent-maps}
Every $\phi\in\cC_e(C)$ is a finite sum of maps
$(F_*^et^i)\cdot\widetilde\alpha$, with $i\in\Z$ and
$\alpha\in\cC_e(B)$.
\end{lemma}

\begin{proof}
Put $q=p^e$. For $0\leq r<q$,
\[
 \phi_r:F_*^eB\longrightarrow C,\qquad
 F_*^eb\longmapsto\phi(F_*^e(bt^r))
\]
is $B$-linear. Choose finitely many $B$-module generators of $F_*^eB$.
Only finitely many powers of $t$ occur in their images under $\phi_r$.
Taking coefficients therefore gives maps $\alpha_{r,k}\in\cC_e(B)$
such that $\phi_r=\sum_kt^k\alpha_{r,k}$, with only finitely many
nonzero terms.

We then have
\[
 \phi=\sum_{r,k}(F_*^et^{qk-r})\cdot
                    \widetilde{\alpha_{r,k}}.
\]
Indeed, on $F_*^e(bt^{qj+r})$, only the summands indexed by $r$
survive. Their sum is
$\sum_kt^{j+k}\alpha_{r,k}(F_*^eb)=t^j\phi_r(F_*^eb)$.
This is also the value of $\phi$, by $C$-linearity.
\end{proof}

For a Cartier subalgebra $\cD\subseteq\cC(B)$, let
$\cD[t]\subseteq\cC(C)$ consist, in degree $e$, of finite sums of maps
$(F_*^eu)\cdot\widetilde\alpha$, with $u\in C$ and $\alpha\in\cD_e$.
In degree $0$, these are just multiplication maps by elements of $C$.
Equations \eqref{eq:laurent-multiplicative} and \eqref{eq:source-rule}
show that these sums are closed under products. Thus $\cD[t]$ is a
Cartier subalgebra. Lemma~\ref{lem:laurent-maps} gives
$\cC(B)[t]=\cC(C)$. Likewise, \eqref{eq:principal-piece} and
\eqref{eq:laurent-multiplicative} give
\begin{equation}\label{eq:laurent-principal}
 B\langle\phi\rangle[t]=C\langle\widetilde\phi\rangle
 \qquad(0\neq\phi\in\cC_s(B),\ s>0).
\end{equation}
Compare also \cite[Corollary~6.10]{BlickleStabler2019},
 \cite[proof of Proposition~6.1]{rodríguezvillalobos2026bcmthresholdsnonprincipalideals} and
\cite{AjitSimper2025}.

\begin{lemma}[{cf.~\cite[Theorem~4.1]{BravoSmith2002}}]
\label{lem:laurent-test-ideal}
Let $\cD\subseteq\cC(B)$ be a nondegenerate Cartier subalgebra. Then
$\tauB(C;\cD[t])=\tauB(B;\cD)C$. In particular,
$\tauB(C)=\tauB(B)C$, and
$\tauB(C;C\langle\widetilde\phi\rangle)
=\tauB(B;B\langle\phi\rangle)C$
for $0\neq\phi\in\cC_s(B)$, $s>0$.
\end{lemma}

\begin{proof}
Put $I=\tauB(B;\cD)$ and $J=\tauB(C;\cD[t])$. $\cD[t]$ is nondegenerate because
$\widetilde\alpha\neq0$ whenever $\alpha\neq0$.

First, $IC$ is $\cD[t]$-stable. Indeed, for $\alpha\in\cD_e$,
$q=p^e$ and $b_n\in I$, we have
\[
 \widetilde\alpha\Bigl(F_*^e\sum_nb_nt^n\Bigr)
 =\sum_{q\mid n}\alpha(F_*^eb_n)t^{n/q}\in IC.
\]
Same holds for source multiples and sums of these maps, since
$IC$ is an ideal. Minimality of $J$ therefore gives $J\subseteq IC$.

For reverse inclusion, $J\cap B$ is $\cD$-stable, since
$\widetilde\alpha(F_*^eb)=\alpha(F_*^eb)$ for $b\in B$.
Thus it suffices to find a nonzero element of $J\cap B$.

Take $0\neq g=\sum_{n=n_0}^{n_1}b_nt^n\in J$ and choose $j$ with
$b_j\neq0$. Iterating a nonzero map in $\cD$ lets us
choose $e>0$ with $\cD_e\neq0$ and $q=p^e>n_1-n_0$
(Lemma~\ref{lem:product-divisor}).
Choose $0\neq\alpha\in\cD_e$ and $y\in B$ such that
$\alpha(F_*^ey)\neq0$.
In $t^{-j}g$, constant term is the only term whose exponent is
divisible by $q$. Hence
\[
 \widetilde\alpha\bigl(F_*^e(b_j^{q-1}yt^{-j}g)\bigr)
 =\alpha(F_*^e(b_j^qy))
 =b_j\alpha(F_*^ey)\neq0.
\]
The input lies in $J$, so this element lies in $J\cap B$.
Minimality of $I$ now gives $I\subseteq J\cap B$. Hence $IC\subseteq J$.

The last assertions follow from Lemma~\ref{lem:laurent-maps} and
\eqref{eq:laurent-principal}.
\end{proof}

\subsection{Boundary on \texorpdfstring{$\Proj S$}{Proj S}}

Choose a Cartier divisor $H$ on $Y$ with
$\cO_Y(H)\cong\cO_Y(1)$. Let $\psi\in\cC_e(S)$ be nonzero and
homogeneous.
Since $\psi$ takes some input of degree $n/q$ to a nonzero element
of integer degree, $\deg\psi=a/q$ for some $a\in\Z$.
To obtain a map into $\cO_Y$, we therefore restrict the source to
degree $-a$.

\begin{proposition}\label{prop:Proj-descent}
There is a nonzero $\cO_Y$-linear map
$\overline\psi:F_*^e\cO_Y(-aH)\to\cO_Y$
whose expression on $D_+(f)=\Spec B$, where $0\neq f\in S_1$,
$B=(S_f)_0$, is
\[
 \phi_f:F_*^eB\longrightarrow B,\qquad
 \phi_f(F_*^eb)=\psi_f(F_*^e(bf^{-a})).
\]
Its boundary $\Delta_Y$ satisfies
$(q-1)(K_Y+\Delta_Y)\sim aH$ and
$\Delta_Y|_{D_+(f)}=\Delta_{\phi_f}$.
If
$\tauB(S;S\langle\psi\rangle)\cO_U=\tauB(S)\cO_U$,
then $\tauB(Y,\Delta_Y)=\tauB(Y)$.
\end{proposition}

\begin{proof}
On $D_+(f)$, sections of $\cO_Y(-a)$ are
$(S_f)_{-a}=B\cdot f^{-a}$
\cite[Tag~\stacktag{01MT}]{stacks-project}.
After applying $F_*^e$, these have degree $-a/q$. So $\psi_f$
sends them into the degree $0$ part $B$.
These restrictions agree on overlaps because they come from the same
map $\psi$. They therefore glue to $\overline\psi$, with local
expression $\phi_f$.

We can recover $\psi_f$ from $\phi_f$.
By \eqref{eq:laurent-chart}, $C=S_f=B[f,f^{-1}]$.
For $b\in B$ and $n\in\Z$, the value
$\psi_f(F_*^e(bf^n))$ has degree $(n+a)/q$. It is zero unless
$n=qk-a$ for some $k\in\Z$. In that case,
\[
 \psi_f(F_*^e(bf^{qk-a}))
 =f^k\psi_f(F_*^e(bf^{-a}))
 =f^k\phi_f(F_*^eb).
\]
Thus, taking $t=f$ in \eqref{eq:Laurent-residue-map}, we get
\begin{equation}\label{eq:psi-chart}
 \psi_f=(F_*^ef^a)\cdot\widetilde{\phi_f}.
\end{equation}
In particular, $\phi_f\neq0$ because $\psi_f\neq0$. Thus
$\overline\psi\neq0$.

Now we compute the boundary
(cf.~\cite[Remark~9.5]{SchwedeFAdjunction2009}).
On the regular locus, Frobenius duality (\cite[Corollary~1.0.3]{BhattBlickleSchwedeTucker2026}) and
projection formula give
\[
 \begin{aligned}
 \mathcal{H}om(F_*^e\cO_Y(-aH),\cO_Y)
 &\cong
 (F_*^e\cO_Y(K_Y+aH))\otimes\cO_Y(-K_Y)\\
 &\cong F_*^e\cO_Y((1-q)K_Y+aH).
 \end{aligned}
\]
Both sides are reflexive. So this isomorphism extends to whole $Y$
\cite[Proposition~2.4(2)--(4)]{SchwedeSmith2010}.
Thus $\overline\psi$ determines an effective divisor
$D_{\overline\psi}\sim(1-q)K_Y+aH$.
Set $\Delta_Y=D_{\overline\psi}/(q-1)$. Then
$(q-1)(K_Y+\Delta_Y)\sim aH$.
In the trivialization by $f^{-a}$, the map is $\phi_f$, so
$\Delta_Y|_{D_+(f)}=\Delta_{\phi_f}$.

Finally, assume the equality of test ideals on $U$.
Since $f^a$ is a unit in $C$, \eqref{eq:psi-chart} gives
\[
 C\langle\psi_f\rangle=C\langle\widetilde{\phi_f}\rangle
\]
as each map is a source multiple of the other by a unit.
Localizing the hypothesis at $f$ and using \eqref{eq:localization}
and Lemma~\ref{lem:laurent-test-ideal}, we obtain
\[
 \begin{aligned}
 \tauB(B;B\langle\phi_f\rangle)C
 &=\tauB(C;C\langle\widetilde{\phi_f}\rangle)\\
 &=\tauB(C;C\langle\psi_f\rangle)
 =\tauB(C)
 =\tauB(B)C.
 \end{aligned}
\]
For every ideal $L\subseteq B$, we have $LC\cap B=L$, by taking
Laurent coefficients. Contracting the above equality therefore
gives $\tauB(B;B\langle\phi_f\rangle)=\tauB(B)$.
By \eqref{eq:principal-comparison},
this is
$\tauB(B,\Delta_Y|_{D_+(f)})=\tauB(B)$.
The charts $D_+(f)$ cover $Y$, proving the assertion.
\end{proof}
\subsection{Proof of Theorem~\ref{thm:intro-global}}

A homogeneous coordinate ring need not be normal even if its Proj is normal \footnote{Twisted quartic curve in $\mathbb P^3$ is an example
\cite[Exercise~I.3.18(b)]{Hartshorne1977}}. So we use a complete section ring.

\begin{lemma}[{cf.~\cite[Chapter~3, Proposition~4.13]{SchwedeSmithBook2026}}]
\label{lem:normal-graded-presentation}
Let $A$ be an excellent Noetherian ring and let $X$ be a normal integral
scheme quasi-projective over $A$. Then there is a normal integral projective
compactification $X\subseteq Y=\Proj S$, where $S$ is a standard graded
normal domain of finite type over $A$ and $\cO_Y(1)$ is ample and
invertible. If $A$ is $F$-finite, then so is $S$.
\end{lemma}

\begin{proof}
Choose an immersion $X\hookrightarrow\mathbb P^n_A$. Let $Z$ be its
scheme-theoretic closure and $Y\to Z$ be the normalization.
The scheme $Z$ is excellent and hence Nagata
\cite[Tags~\stacktag{07QU} and~\stacktag{07QV}]{stacks-project}.
So $Y\to Z$ is finite
\cite[Tag~\stacktag{035S}]{stacks-project}.
Thus $Y$ is projective over $A$
\cite[Tags~\stacktag{0B3I} and~\stacktag{0C4P}]{stacks-project}.
Since $X$ is normal, normalization is an isomorphism over $X$
\cite[Tag~\stacktag{0AB1}]{stacks-project}.
View $X$ as a dense open subscheme of the normal integral
scheme $Y$.

Choose an ample invertible sheaf $L$ on $Y$ and put
\[
 T=\bigoplus_{n\geq0}H^0(Y,L^n).
\]
By \cite[Chapter~3, Proposition~4.13(a),(b),(d)]{SchwedeSmithBook2026},
$T$ is a normal domain, finitely generated over its degree $0$ part,
which is a finite $A$-algebra. We have $\Proj T=Y$.
Choose $d>0$ sufficiently divisible so that
\[
 S=T^{(d)}=\bigoplus_{n\geq0}H^0(Y,L^{dn})
\]
is generated in degree $1$
\cite[Tag~\stacktag{0EGH}]{stacks-project}.
This is the section ring of $L^d$. So $S$ is normal and of finite type over $A$, with
$\Proj S=Y$ and $\cO_Y(1)\cong L^d$.

Finally, if $A$ is $F$-finite, absolute Frobenius on $\Spec S$ factors
as relative Frobenius followed by a base change of Frobenius on
$\Spec A$. Both are finite
\cite[Tag~\stacktag{0CCD}]{stacks-project}. So $S$ is $F$-finite.
\end{proof}

\begin{proof}[Proof of Theorem~\ref{thm:intro-global}]
Suppose first that $X$ is integral and $\dim X>0$.
Base is $F$-finite, hence excellent \cite{Kunz1976}.
Lemma~\ref{lem:normal-graded-presentation} gives a compactification
$X\subseteq Y=\Proj S$, with $S$ standard graded, normal and $F$-finite.
Take
$K_X=K_Y|_X$.

Since $\dim Y>0$, Proposition~\ref{prop:homogeneous-interpolation}
gives a nonzero homogeneous $\psi\in\cC_s(S)$ such that
\[
 \tauB(S;S\langle\psi\rangle)\cO_U=\tauB(S)\cO_U.
\]

By \Cref{rem:a-positive} we may choose $\deg\psi>0$.
Write $\deg\psi=a/p^s$, with $a>0$. Choose a Cartier divisor $H$ with
$\cO_Y(H)\cong\cO_Y(1)$. Then $H$ is ample.
Proposition~\ref{prop:Proj-descent} gives effective $\Delta_Y$ with
\[
 \tauB(Y,\Delta_Y)=\tauB(Y),\qquad
 (p^s-1)(K_Y+\Delta_Y)\sim aH.
\]
Restricting to $X$ and using \eqref{eq:localization} proves the assertion
with $\Delta=\Delta_Y|_X$. The index of $K_X+\Delta$ divides $p^s-1$. So prime to $p$.

If $X$ is integral of dimension $0$, it is a field,
and $\Delta=0$ works.

Finally, a normal Noetherian scheme is a finite disjoint union of
normal integral components
\cite[Tag~\stacktag{0357}]{stacks-project}.
Each component is open and closed, hence quasi-projective over the
base. So the integral cases apply. Their boundaries define an effective
$\Q$-divisor on $X$. The equality of test ideals holds componentwise.
Global index is the lcm of the component indices. So prime to $p$.
\end{proof}

\section{Applications}\label{sec:consequences}

Here we give some applications of our main results. 

\subsection{Fixed ideals and exponents}

For mixed test ideals, see \cite[Section~1]{Takagi2006}. We use their
description by Cartier algebras
\cite[Remarks~3.10 and~3.17]{Schwede2011},
\cite[Definitions~3.20--3.21, Proposition~3.22 and the discussion
following it]{Blickle2013}.

\begin{corollary}[{cf.~\cite[Remark~3.10]{Schwede2011}}]\label{cor:intro-mixed}
Let $R$ be a normal $F$-finite domain, let
$\mathfrak a_1,\ldots,\mathfrak a_r$ be nonzero ideals and fix real
numbers $t_1,\ldots,t_r\geq0$. There is an effective $\Q$-divisor
$\Delta$ such that $K_R+\Delta$ is $\Q$-Cartier of index prime to $p$
and
\[
 \tauB(R;\mathfrak a_1^{t_1}\cdots\mathfrak a_r^{t_r})
 =\tauB(R,\Delta).
\]
\end{corollary}

\begin{proof}
We put the ideals and exponents into one Cartier algebra. Set
$\mathfrak b_0=R$ and, for $e\geq1$, put
\[
 \mathfrak b_e
 =\prod_{i=1}^r\mathfrak a_i^{\ceil{t_i(p^e-1)}}.
\]
Let $\cD_e=\cC_e(R)\cdot F_*^e\mathfrak b_e$ be the set of finite
sums of maps $(F_*^eb)\cdot\phi$, where $b\in\mathfrak b_e$ and
$\phi\in\cC_e(R)$. Put $\cD_0=R$.

For each $i$,
\[
 p^f\ceil{t_i(p^e-1)}+\ceil{t_i(p^f-1)}
 \geq\ceil{t_i(p^{e+f}-1)}.
\]
So,
\begin{equation}\label{eq:mixed-coefficients}
 u^{p^f}v\in\mathfrak b_{e+f}
 \qquad
 (u\in\mathfrak b_e,\ v\in\mathfrak b_f).
\end{equation}
Together with \eqref{eq:source-rule}, this shows that $\cD$ is a
Cartier subalgebra. Its test ideal is $\tauB(R;\mathfrak a_1^{t_1}\cdots\mathfrak a_r^{t_r})$. Also $\cD_1\neq0$.
Corollary~\ref{thm:intro-principalization} therefore gives the assertion.
\end{proof}

Corollary~\ref{cor:intro-mixed} absorbs ideals into the boundary. With some more careful work, we can also keep them in the triple as explained below. \begin{lemma}\label{lem:keep-ideals}
Let $R$ be a normal $F$-finite domain. Fix nonzero ideals
$\mathfrak a_1,\ldots,\mathfrak a_r\subseteq R$ and real numbers
$t_1,\ldots,t_r\geq0$. There is an effective $\Q$-divisor
$\Delta$, with $K_R+\Delta$ $\Q$-Cartier of index prime to $p$, such
that
\[
 \tauB(R;\mathfrak a_1^{t_1}\cdots\mathfrak a_r^{t_r})
 =
 \tauB(R,\Delta;\mathfrak a_1^{t_1}\cdots\mathfrak a_r^{t_r}).
\]
\end{lemma}

\begin{proof}
Use $\mathfrak b_e$ and $\cD$ from the preceding proof. Put
$I=\tauB(R;\cD)$. Field case is immediate. Otherwise choose
$0\neq\theta\in\cC_1(R)$, $0\neq g\in\mathfrak b_1$. Put
$\xi=(F_*g)\cdot\theta$. Choose $c_0$ by Lemma~\ref{lem:common-data}
for $R\langle\xi\rangle$, then multiply it into
$\bigcap_i\mathfrak a_i$. It is HDBS for this algebra and for $\cD$.
$\theta_{c_0}$ generates.

By \eqref{eq:hdbs-formula}, generators of $I$ are finite sums of terms
$z_j=\gamma_j(F_*^{d_j}(c_0u_j))$, where
$\gamma_j\in\cC_{d_j}(R)$ and $u_j\in\mathfrak b_{d_j}$.
Interpolate these individual terms.
Dual-basis construction in Proposition~\ref{prop:finite-interpolation}
gives, after replacing $c_0$ by a nonzero multiple $c$,
\[
 \beta_i(w_j)=\delta_{ij}c^L,\qquad
 \beta_i\in\cC_e(R),\quad w_j\in F_*^e(c^2\mathfrak b_e).
\]
That extra factor $\mathfrak b_e$ is obtained by clearing denominators,
since $(\mathfrak b_e)_c=R_c$.
We have $c=\theta(F_*(ch))$ with $h\in\mathfrak b_1$.
The map $(F_*(c^{(p-1)(L-1)}h))\cdot\theta$ fixes $c^L$.
Since $c_0$ is HDBS for $R\langle\xi\rangle$, proof of
Lemma~\ref{lem:synchronization}, using this fixing map, gives
\[
 z_j=\alpha_j(F_*^D(c^Lv_j)),\qquad
 \alpha_j\in\cC_D(R),\quad v_j\in\mathfrak b_D
\]
in one degree $D>0$. All coefficients remain in the required ideals by
\eqref{eq:mixed-coefficients}.

Set $\psi_0=\sum_j\alpha_j\st\beta_j$ and $s_0=D+e$. Then
\[
 \psi_0(F_*^D(v_jw_j))=z_j,\qquad
 v_jw_j\in F_*^e(c^2v_j^{p^e}\mathfrak b_e).
\]
Thus $I\subseteq\psi_0(F_*^{s_0}(c^2\mathfrak b_{s_0}))$.
Put $\psi=\psi_0\st\theta$ and $s=s_0+1$.
Lemma~\ref{lem:compose-theta} makes the localized algebra strongly
$F$-regular. Moreover,
$\theta(F_*(c^{p+1}hr^p))=c^2r$ for $r\in\mathfrak b_{s_0}$.
Equation~\eqref{eq:mixed-coefficients} therefore gives
\[
 I\subseteq\psi(F_*^s(c^2\mathfrak b_s)).
\]

Let $\cA_m=(R\langle\psi\rangle)_m\cdot F_*^m\mathfrak b_m$.
Then $\cA\subseteq\cD$ and
$\cA_c=R_c\langle\psi_c\rangle$.
Also $c\in\psi(F_*^s(c\mathfrak b_s))$.
Induction in Lemma~\ref{lem:principal-test-element} now gives
\[
 c\in\psi^{[n]}(F_*^{ns}(c\mathfrak b_{ns})),\qquad
 c^2\in\psi^{[n]}(F_*^{ns}(c^{p^{ns}+1}\mathfrak b_{ns})).
\]
For any $0\neq d\in R$, strong $F$-regularity after inverting $c$
puts some power $c^k$ in
$\psi^{[a]}(F_*^{as}(d\mathfrak b_{as}))$.
Substitute this into the second formula with $p^{ns}+1\geq k$.
Equation~\eqref{eq:mixed-coefficients} keeps the resulting coefficient
in $\mathfrak b_{(n+a)s}$. Hence $c^2$ is HDBS for $\cA$. We have
\[
 I\subseteq\psi(F_*^s(c^2\mathfrak b_s))
 \subseteq\tauB(R;\cA)\subseteq I.
\]
The $s$th Veronese of $\cA$ is that of the Cartier algebra for
$(R,\Delta_\psi;\prod_i\mathfrak a_i^{t_i})$.
Lemma~\ref{lem:veronese} identifies their test ideals; compare
\cite[Remarks~3.10,~3.17 and Corollary~4.8]{Schwede2011}.
Taking $\Delta=\Delta_\psi$ proves the assertion.
\end{proof}

Thus the version of \cite[Question~6.3]{Schwede2011} with big test
ideals on both sides has an affirmative answer in this affine
setting. The original question has finitistic test ideal on the
right.

The boundary may depend on the ideals and exponents. In fact, the
irrational/indiscrete examples of \cite[Theorem~2]{AjitJumping2026}, together
with \cite[Main Theorem]{BlickleSchwedeTakagiZhang2010}, show that
even for one fixed ideal a single boundary need not work for all
exponents.

\subsection{Finite covers and alterations}

By \cite[Theorem~4.1]{SchwedeTuckerZhang2012}, one finite separable
cover computes the test ideal of a pair $(X,\Delta)$ with
$K_X+\Delta$ $\Q$-Cartier. Our boundary therefore gives such a
description of $\tauB(X)$. Also see  \cite[Theorem~5.6]{ChiecchioEnescuMillerSchwede2018} under finite generation of the
anti-canonical algebra.

We use $\Tr_\pi$ for the trace map from duality, with compatible
canonical divisors as in \cite[\S2]{SchwedeTuckerZhang2012}.
For separable $\pi$, it is induced by the field trace.
For real $\lambda\geq0$, the test ideal
$\tauB(X,\Delta+\lambda D)$ is defined by the Cartier algebra whose
degree $e$ maps are
\[
 \mathcal{H}om_{\cO_X}
 \bigl(F_*^e\cO_X(\ceil{(p^e-1)(\Delta+\lambda D)}),\cO_X\bigr),
\]
viewed as maps on $F_*^e\cO_X$ by restriction, just as in
Section~\ref{sec:map-divisor}; see also
\cite[\S2]{SchwedeTuckerZhang2012}.

\begin{corollary}[{cf.~\cite[Theorem~4.1]{SchwedeTuckerZhang2012}}]\label{cor:one-cover}
Let $X=\Spec R$ be normal, integral and $F$-finite, choose $\Delta$
as in Corollary~\ref{thm:intro-affine} and fix an effective Cartier
divisor $D$ on $X$. There is a finite separable dominant morphism
$\pi:Y\to X$, with $Y$ normal and integral, such that, for every real
$\lambda\geq0$,
\begin{equation}\label{eq:one-cover-ray}
 \tauB(X,\Delta+\lambda D)
 =
 \im\left[
 \pi_*\cO_Y\bigl(
 \ceil{K_Y-\pi^*(K_X+\Delta+\lambda D)}
 \bigr)
 \xrightarrow{\ \Tr_\pi\ }K(X)
 \right].
\end{equation}
At $\lambda=0$, this image is $\tauB(X)$.
\end{corollary}

\begin{proof}
Apply \cite[Theorem~4.1]{SchwedeTuckerZhang2012} to $(X,\Delta)$
and $D$. The finite separable cover exists for arbitrary normal
$F$-finite $X$, as noted immediately after the Main Theorem there.
At $\lambda=0$, use $\tauB(X,\Delta)=\tauB(X)$.
\end{proof}

If $X$ is of finite type over an $F$-finite field,
\eqref{eq:one-cover-ray} also holds for a suitable regular alteration.
Over a perfect field, this alteration can be chosen separable
\cite[Theorem~4.1]{SchwedeTuckerZhang2012}.
The jumping numbers of
$\lambda\mapsto\tauB(X,\Delta+\lambda D)$ are discrete and rational
by \cite[Theorem~3.1]{SchwedeTuckerZhang2012}; see also
\cite{BlickleSchwedeTakagiZhang2010}.

Applying \cite[Theorem~4.1]{SchwedeTuckerZhang2012}, with $D=0$,
to the boundary of Corollary~\ref{cor:intro-mixed} gives a finite
separable cover whose trace computes the fixed mixed test ideal
$\tauB(R;\mathfrak a_1^{t_1}\cdots\mathfrak a_r^{t_r})$.
For varying $\lambda$, however, \eqref{eq:one-cover-ray} computes
$\tauB(X,\Delta+\lambda D)$, which need not equal
$\tauB(X;\cO_X(-D)^\lambda)$.
For test ideals of general ideals, rationality can fail without the
$\Q$-Cartier hypothesis \cite[Theorem~2]{AjitJumping2026}.

\appendix

\section{Playing with Products}\label{App}

\begin{center}
    \textit{Joint with Daniel Apsley.}
\end{center}

A natural question following the positive result of Corollary~\ref{thm:intro-affine} is whether the $F$-signature of a non-$\mathbb Q$-Gorenstein singularity $\operatorname{Spec}R$ can be approximated by the $F$-signatures of appropriate log pairs $(\operatorname{Spec}R,\Delta)$ with $K_R+\Delta$ $\mathbb Q$-Cartier of index prime to the characteristic. This question was asked by Polstra and Tucker in \cite[Question~7.5]{PolstraTucker2018}. In this appendix, we exhibit a simple example giving a negative answer. We also describe a proposed counterexample to another natural question about lifting the result of Corollary~\ref{thm:intro-affine} to characteristic zero. 

\subsection{A Question of Polstra and Tucker} Throughout this subsection, fix an algebraically closed field $k$ of characteristic $p>0$ and let
\[
    S=\bigoplus_{d\geq0}H^0(\mathbb P^1_k\times\mathbb P^2_k,
                            \mathcal O(d,d)).
\]
Let $(R,\mathfrak m)$ be the local ring of $\operatorname{Spec}S$ at the cone point. Then $R$ is an $F$-finite normal local domain of dimension $4$.

\begin{lemma}\label{lem:gap-basic-revised}
\begin{enumerate}
    \item $R$ is not $\mathbb Q$-Gorenstein.
    \item $s(R)=11/24$; in particular, $R$ is strongly $F$-regular.
    \item For every integer $q\geq1$,
    \[
        \ell_R(R/\mathfrak m^q)
        =\frac{q(q+1)(q+2)(3q+1)}{24}.
    \]
\end{enumerate}
\end{lemma}

\begin{proof}
First assertion follows from the fact that the polarization $\mathcal O(1,1)$ is not rationally proportional to the canonical bundle $\mathcal O(-2,-3)$ in the N\'eron--Severi space of $\mathbb P^1\times\mathbb P^2$; see \cite[Proposition~3.14]{kollarsings}. Second and third assertions are toric computations. We direct towards \cite[Example~7]{Singh2005} (compare
\cite[Theorem~3.3]{vonKorff2011}) and leave the details to the reader.
\end{proof}

\begin{theorem}\label{thm:gap-revised}
If $\Delta$ is an effective $\mathbb Q$-divisor on $\operatorname{Spec}R$ such that $K_R+\Delta$ is $\mathbb Q$-Cartier of index prime to $p$, then $s(R,\Delta)\leq1/4$. In particular,
\[
    s(R)-s(R,\Delta)\geq\frac{5}{24}.
\]
\end{theorem}

\begin{proof}
Let $N$ be the index of $K_R+\Delta$ and choose $e_0>0$ such that $N\mid p^{e_0}-1$. For $e$ divisible by $e_0$, put $q=p^e$. Write $\mathscr C^\Delta_\bullet$ for the Cartier subalgebra associated to $\Delta$ and $I_e^\Delta$ for its splitting ideals. Since $(q-1)(K_R+\Delta)$ is Cartier, $\mathscr C^\Delta_e$ is generated by a single map $\varphi_e$ under precomposition with multiplication. $\mathfrak m^{[q]}\subseteq I_e^\Delta$ shows that
$A_e=R/I_e^\Delta$ has finite length. It is either zero or Artinian
Gorenstein. Indeed, the $k$-linear functional
\[
    \lambda(f)=\bigl(\overline{\varphi_e(F^e_*f)}\bigr)^q
\]
induces the perfect pairing $(\bar f,\bar g)\mapsto\lambda(fg)$ on $A_e$, by the definition of $I_e^\Delta$.

The usual splitting ideal $I_e$ is monomial. Toric description shows that every monomial outside $I_e$ has Segre degree at most $2(q-1)$; see \cite[Theorem~3.3 and its proof]{vonKorff2011}. Thus $\mathfrak m^{2q-1}\subseteq I_e$. Since $\Delta$ is effective, we obtain
\[
    \mathfrak m^{2q-1}\subseteq I_e\subseteq I_e^\Delta.
\]
Let $J$ be the image of $\mathfrak m^q$ in $A_e$. Then $J^2=0$, so $J\subseteq\operatorname{Ann}_{A_e}(J)$. If $A_e\ne0$, the Gorenstein property gives a perfect pairing $A_e\times A_e\to k$ and the inclusion above yields
\[
    \ell_R(J)\leq\ell_R(\operatorname{Ann}_{A_e}(J))
               =\ell_R(A_e/J).
\]
Consequently, whether or not $A_e=0$,
\[
    \ell_R(A_e)\leq2\ell_R(A_e/J)
    =2\ell_R\bigl(R/(I_e^\Delta+\mathfrak m^q)\bigr)
    \leq2\ell_R(R/\mathfrak m^q).
\]
Divide by $q^4$ and let $e\to\infty$ through multiples of $e_0$. By existence  \cite{PolstraTucker2018} and Lemma~\ref{lem:gap-basic-revised}(3),
\[
    s(R,\Delta)
    \leq\lim_{q\to\infty}\frac{q(q+1)(q+2)(3q+1)}{12q^4}
    =\frac14.
\]
\end{proof}

\begin{remark}
Since the toric computations above are independent of the characteristic, the same example gives a gap for the limit $F$-signature. Indeed, any fixed effective $\mathbb Q$-divisor $D$ on the corresponding characteristic-zero cone with $K_R+D$ $\mathbb Q$-Cartier has reductions satisfying the hypotheses of Theorem~\ref{thm:gap-revised} in sufficiently large characteristics. 
\end{remark}

\subsection{Further Questions and Counterexamples} 
Here, we consider variations of the above to give negative answers to questions about choosing boundaries uniformly. For example, we first observe that one cannot always choose a single boundary $\Delta$ with $K_R+\Delta$ $\mathbb Q$-Cartier of index prime to $p$ that computes the test ideals of $(R, \mathfrak{a}^\lambda)$ for every $\lambda \geq 0$, even when anticanonical algebra of $R$ is finitely generated. 

\begin{remark}
The following examples were suggested by ChatGPT 6.1 Sol during our further investigation of the example above. Both involve at least one nontrivial step requiring substantial calculation. We record them here (without claiming any credit) in the hope that simpler counterexamples can be found. We have not tried that. 
\end{remark}

 \begin{question}\label{q:prime-to-p-uniform-boundary}

Let $R$ be a normal $F$-finite domain of characteristic $p>0$ whose
anti-canonical algebra is finitely generated and let
$\mathfrak a$ be a nonzero ideal of $R$. Does there always exist an
effective $\mathbb Q$-divisor $\Delta$ such that $K_R+\Delta$ is
$\mathbb Q$-Cartier of index prime to $p$ and
\[
\tauB(R;\mathfrak a^\lambda)
=
\tauB(R,\Delta;\mathfrak a^\lambda)
\]
for every real $\lambda\geq0$?
\end{question}


\begin{example}\label{example-AI-1}(ChatGPT)
    Let $k = \mathbb{F}_5$ and let $\operatorname{Spec} R$ be the affine cone over $\mathbb{P}^1 \times C$ with respect to $\mathcal{O}(1) \boxtimes \omega_C$, where $C$ is the genus $2$ curve in $\mathbb{P}(1,1,3)$ defined by the equation $$y_0^6 + y_1^6 = y_2^2.$$ Then, $R$ is normal, $F$-finite and its anticanonical algebra is finitely generated. Indeed, $R$ is the Segre product of $k[x_0, x_1]$ and the weighted homogeneous coordinate ring of $C$. Its anticanonical algebra is generated over $R$ by the four binary cubics in anticanonical degree $1$.

    Consider the $\Q$-Weil divisor $\Delta = \frac{1}{10}D_f$ on $\operatorname{Spec} R$, where $D_f$ is the cone over the divisor cut out by $$f = x_0^{30} + 4x_1^{14}x_0^{16} + x_1^{23}x_0^7 + 3x_1^{30}$$ on $\mathbb{P}^1 \times C$. Writing $\overline{\Delta}$ for the corresponding divisor on $\mathbb{P}^1 \times C$, we have $10(K_{\mathbb{P}^1 \times C} + \overline{\Delta}) \sim \mathcal{O}_{\mathbb{P}^1}(10) \boxtimes \omega_C^{\otimes 10}$. Thus $K_R + \Delta$ is $\mathbb{Q}$-Cartier of index dividing $10$.

    Moreover, one can show (by a substantial calculation) that if $g = x_0y_0 - x_1y_1 \in R$ and $I = R_{\geq 2}$, then $$\tauB(R; (g)^\lambda) = \tauB(R, \Delta; (g)^\lambda) = g^{\lfloor\lambda\rfloor}I \qquad \text{for every real }\lambda \geq 0.$$

    On the other hand, there is no effective $\Q$-divisor $\Delta'$ such that $K_R + \Delta'$ is $\Q$-Cartier of index coprime to $5$ and 

    \[
\tauB(R;(g)^\lambda)
=
\tauB(R,\Delta';(g)^\lambda)
\qquad
\text{for every real }\lambda\geq0.
\]
The nonexistence of $\Delta'$ follows from the formula above, \cite[Lemma 2.1]{stabler} and a graded Frobenius calculation as $\lambda$ approaches $1$ from below. 
\end{example}

Next, we quote a related example which we expect to give a negative answer to the following natural liftability question. See \cite[Theorem D]{ChiecchioEnescuMillerSchwede2018} and Corollary~\ref{thm:intro-affine} for motivation. 

\begin{question}
    Let $R$ be a normal domain essentially of finite type over a field of characteristic $0$, with finitely generated anticanonical algebra and fix a nonzero ideal $\mathfrak{a}$ of $R$. After spreading out $(R,\mathfrak a)$ over a finitely generated integral $\mathbb Z$-algebra, does there exist an effective $\mathbb Q$-divisor $\Delta$ with $K_R+\Delta$ $\mathbb Q$-Cartier and a dense open subset of the arithmetic base such that
\[
\tauB(R_s;\mathfrak a_s^\lambda)
=
\tauB(R_s,\Delta_s;\mathfrak a_s^\lambda)
\]
for every real $\lambda\geq0$ and every closed point $s$ in this open set?

\end{question}

\begin{example}\label{example-AI-2}(ChatGPT)
    Let $R$ be the local ring of the vertex of the affine cone over $\mathbb{P}^1 \times \mathbb{P}^2$ with respect to $\mathcal{O}(1,2)$ over the complex numbers. We can describe this local ring explicitly as the localization at the vertex of the invariant subring of $\mathbb{C}[x,z,y_1, y_2, y_3]$ where $\mathbb{C}^*$ acts on this polynomial ring with weights $(2,2,-1,-1,-1)$. We fix $c \in \mathbb C$ transcendental over $\Q$ and define $$g = [xz(x -z)(x - cz)]^3 (y_1y_2y_3)^8, \;\;\;\;\;\;\; \mathfrak{a} = (g) \subset R.$$ This ring is toric, non-$\Q$-Gorenstein and its anticanonical algebra is finitely generated by Gordan's lemma.

    Spreading out $c$ as a parameter $t$ gives a model of $R$ over $\Z[t,t^{-1},(t-1)^{-1}]$. For each closed fiber $s$ of characteristic $p \geq 5$, one has $$\operatorname{fpt}(R_s, g_s) = \frac{1}{8}.$$

    To obtain a negative answer, one would need to show that for every fixed effective characteristic-zero boundary $\Delta$ with $K_R+\Delta$ $\Q$-Cartier, every dense open subset of a base over which $(R,g,\Delta)$ is defined contains closed points $s$ of arbitrarily large positive characteristic for which $$\operatorname{fpt}(R_s, \Delta_s, g_s) < \frac{1}{8} = \operatorname{fpt}(R_s, g_s).$$
    This then follows from a calculation of the degree of the Hodge bundle for a family of cyclic covers associated to $\Delta$. 
\end{example}

\subsection{Future Directions} Finally, given this gap in the $F$-signature, we ask whether similar behavior occurs for other local measures of singularities. 

\begin{definition}
    Let $(R,\mathfrak m)=\mathcal O_{X,x}$ be the local ring of a potentially klt singularity $x \in X$ of dimension $n$ in characteristic $0$. In this setting, $R$ is log terminal in the sense of \cite{deFernexHacon2009}. Using the log canonical threshold of \cite[Definition 7.1]{deFernexHacon2009}, we define $$\widehat{\operatorname{vol}}(x,X)=\widehat{\operatorname{vol}}(R,\mathfrak m) := \inf_{\mathfrak a} \operatorname{lct}(R,\mathfrak a)^n e(\mathfrak a),$$ where the infimum is over $\mathfrak m$-primary ideals and $e(\mathfrak a)$ denotes the Hilbert--Samuel multiplicity. 
\end{definition}

\begin{remark}
    When $R$ is $\Q$-Gorenstein, this agrees with its usual normalized volume by \cite[Theorem 27]{liu18}.
\end{remark}

\begin{question}
    Is there an equality $$\widehat{\operatorname{vol}}(x, X) = \sup_{\Delta} \widehat{\operatorname{vol}}(x, X, \Delta),$$ where $\Delta$ ranges over effective $\Q$-divisors such that $x \in (X,\Delta)$ is klt? If so, is this supremum a maximum? Does either statement hold when the normalized volume is replaced by the $F$-volume in positive characteristic? 
\end{question}

\bibliographystyle{amsalpha}
\bibliography{ArXiv}
\end{document}